\documentclass[draft,reqno]{amsart}

\usepackage{amsmath}
\usepackage{amssymb}
\usepackage[all]{xy}
\usepackage{picinpar}
\usepackage{palatino}
\usepackage[usenames,dvipsnames,svgnames,table]{xcolor}
\usepackage{mathtools}

\def\eu{\mathfrak}
\def\ma{\mathbb}
\def\mc{\mathcal}

\def\p{{\eu p}_{\infty}}
\def\f{{\ma F}_q^{\ast}}
\def\F{{\ma F}_q}
\def\P{\mathcal P}
\def\pK{\eu p}
\def\pL{\eu P}
\def\fin{\hfill\qed\bigskip}
\def\ge#1{#1_{\eu{gex}}}
\def\g#1{#1_{\eu{ge}}}
\def\*#1{#1^*}
\def\cicl#1#2{k(\Lambda_{{#1}^{#2}})}
\def\cic#1#2{{\ma Q}(\zeta_{{#1}^{#2}})}
\def\lra{\longrightarrow}
\def\ew#1#2{e_{\infty}^{\text{wild}}(#1|#2)}

\newcommand{\Gal}{\operatorname{Gal}}
\newcommand{\lcm}{\operatorname{lcm}}

\newcommand{\Id}{\operatorname{Id}}

\newcommand{\im}{\operatorname{im}}
\newcommand{\N}{\operatorname{N}}
\newcommand{\rest}{\operatorname{rest}}
\newcommand{\con}{\operatorname{con}}

\newcommand{\Hom}{\operatorname{Hom}}
\newcommand{\hooklongrightarrow}{\lhook\joinrel\longrightarrow}

\newcounter{bean}
\newcounter{2bean}

\def\l{
\begin{list}
{\rm{(\alph{bean}).-}}{\usecounter{bean}
\setlength{\labelwidth}{0.8in}
\setlength{\labelsep}{0.3cm}
\setlength{\leftmargin}{1cm}}}

\def\las{\begin{list}
	{{\rm {(\arabic{2bean})}}}{\usecounter{2bean}
\setlength{\labelwidth}{0.8in}
\setlength{\labelsep}{0.3cm}
\setlength{\leftmargin}{1cm}}}

\numberwithin{equation}{section}
\newtheorem{theorem}{Theorem}[section]
\newtheorem{proposition}[theorem]{Proposition}

\newtheorem{example}[theorem]{Example}
\newtheorem{remark}[theorem]{Remark}
\newtheorem{definition}[theorem]{Definition}
\newtheorem{corollary}[theorem]{Corollary}

\title[Genus fields of global fields]
{Genus fields of global fields}
 
\author[E. Ram\'irez]{Elizabeth Ram\'irez--Ram\'irez}
\address{Departamento de Control Autom\'atico\\
Centro de Investigaci\'on y de Estudios Avanzados del I.P.N.}
\email{eramirez@ctrl.cinvestav.mx}
 
\author[M. Rzedowski]{Martha Rzedowski--Calder\'on}
\address{Departamento de Control Autom\'atico\\
Centro de Investigaci\'on y de Estudios Avanzados del I.P.N.}
\email{mrzedowski@ctrl.cinvestav.mx}

\author[G. Villa]{Gabriel Villa--Salvador}
\address{Departamento de Control Autom\'atico\\
Centro de Investigaci\'on y de Estudios Avanzados del I.P.N.}
\email{gvillasalvador@gmail.com, gvilla@ctrl.cinvestav.mx}



\date{January 21st., 2019}

\begin{document}

\begin{abstract}

In this paper we obtain the extended genus field of a 
global field. First we define the extended genus field
of a global function field and we obtain, via class field
theory, the description of the extended genus field of 
an arbitrary global function field. In the last part
of the paper we use the techniques for function fields
to describe the extended genus field of an arbitrary
number field.

\end{abstract}

\maketitle

\section{Introduction}\label{S1}

The study of {\em narrow} or {\em
extended genus fields} goes back to C.F. Gauss 
\cite{Ga1801} who introduced the genus concept in the context of
quadratic forms. During the first half of the last century, the
concept was imported to quadratic number fields. H. Hasse
\cite{Ha51} studied genus theory of quadratic number fields
by means of class field theory. H.W. Leopoldt \cite{Le53}
generalized the work of H. Hasse by introducing the concept
of genus field for a finite abelian extension of the rational field. Leopoldt
studied extended genus fields using the arithmetic of abelian fields
by means of {\em Dirichlet characters}. The first to introduce
the concept of genus field and of
extended genus field of a nonabelian finite extension of
the rational field was A. Fr\"olich who defined the concept of
genus field of an arbitrary finite extension of ${\ma Q}$
\cite{Fr59-1,Fr59-2}. For a number field $K$, Fr\"olich defined
the genus field $K$ (with respecto to the rational field ${\ma Q}$)
as $\g K:=KF$ where $F/{\ma Q}$ is the maximum abelian
extension such that $KF/K$ is unramified everywhere. Similarly,
the extended genus field is $\ge K=KL$ where $L/{\ma Q}$
is the maximum abelian extension such that $KL/K$
is unramified at the finite primes. Numerous authors have
studied genus fields and extended genus fields for finite
field extensions $K/{\ma Q}$ over ${\ma Q}$. 

In the case of number fields, the concepts of {\em Hilbert class field}
and of {\em extended Hilbert class field} are defined without any
ambiguity. The Hilbert class field $K_H$ and the extended Hilbert class
field $K_{H^+}$ of a number field $K/{\ma Q}$ are defined as the
maximum abelian unramified extension and the maximum abelian extension
unramified at the finite primes of $K$, respectively. In this way, the concepts of
genus field and of extended genus field are defined depending on the
concept of the Hilbert class field, and of the extended Hilbert class
field respectively. Namely, we have $K\subseteq \g K\subseteq
K_H$ and the Galois group $\Gal(K_H/K)$ is isomorphic
to the class group $Cl_K$ of $K$. The genus field $\g K$ corresponds to a 
subgroup $G_K$ of $Cl_K$ and we have $\Gal(\g K/K)\cong
Cl_K/G_K$. The degree $[\g K:K]$ is called the {\em genus 
number} of $K$ and $\Gal (\g K/K)$ is called the {\em genus group}
of $K$.
Similarly, $K\subseteq \ge K\subseteq K_{H^+}$ and $\ge K$ corresponds
to a subgroup $G_{K^+}$ of $\Gal(\ge K/K)\cong Cl_{K^+}$.

For global function fields the picture is different due to
the fact that there are several
concepts of Hilbert class field and of extended Hilbert class field,
depending in which aspect you are interested in. The direct definition
of the Hilbert class field $K_H$ of a global function field $K$ over
$\F$ as the maximum unramified abelian extension of $K$ has
the disadvantage of being of infinite degree over $K$ due to the extensions
of constants. In the extensions of constants, every prime is
eventually inert, so, if we are interested in a definition of a Hilbert
class field of finite degree over the base field, we must impose
some condition on the extension of constants. It seems
that the first one to consider extended genus fields in the case of
function fields was R. Clement in \cite{Cl92}, where she considered the
case of a cyclic tame extension $K/\F(T)$ of prime degree $l$ different
from the characteristic $p$ of $\F$. She developed the theory along
the lines of the case studied by Hasse in \cite{Ha51}.
Later on, S. Bae and J.K. Koo \cite{BaKo96} generalized the results
of Clement following the development given by Fr\"olich. They defined
the extended genus field for extensions of an arbitrary global
function field $K$ defining an analogue to the cyclotomic function
field extensions of $\F(T)$ given by the Carlitz module.

M. Rosen defined in \cite{Ro87} the Hilbert class field of a global function field
$K$ as the maximum abelian unramified extension of $K$ such that
a fixed nonempty finite set of prime divisors of $K$ decompose
fully. Using this definition of Hilbert class field, G. Peng \cite{Pe2003}
found the genus field of a cyclic tame extension of prime degree over
the rational function field $k=\F(T)$. His method used the
analogue for function fields
of the Conner--Hurrelbrink exact hexagon in number fields. The
wild prime case was presented by S. Hu and Y. Li in \cite{HuLi2010} where
they described explicitly the genus field of an Artin--Schreier extension of the
rational function field. In 
\cite{BaMoReRzVi2018, MaRzVi2013, MaRzVi2017} we developed
a theory of genus fields using the same concept of Hilbert class field.
In those papers, the ideas of Leopoldt using Dirichlet characters were
strongly used.

In this paper we are interested in describing, using class field theory,
the extended genus field of a finite separable extension of $k$. 
B. Angl\`es and J.-F. Jaulent in \cite{AnJa2000} established the general
theory of extended genus fields of global fields, either function
or numeric. We use a concept of extended genus field
for function fields different from the one defined by Angl\`es and Jaulent. With
this concept, when we describe the finite abelian extension $L$
where $\ge K=KL$, we may write $L$ as the composition of a sort
of $P$--components, where $P$ runs through the finite primes
of $k$. We consider these $P$--components $L_P$ 
as the composition of $E_P$,
the $P$--component of the projection $E$ of $L$ in a cyclotomic
function field given by the Carlitz module,
and a field $S$ which codifies the behavior of the infinite
prime. More precisely, $S$ codifies the wild ramification and the inertia
of the infinite prime of $k$. To this end, we need to consider the
id\`ele group corresponding to an arbitrary cyclotomic function
field. Finally, we describe the field $S$.

It turns out, that the same approach works for number fields. Indeed,
in the number field case, the problem is simpler because, by the 
Kronecker--Weber theorem, any abelian extension of ${\ma Q}$
is cyclotomic, that is, it is contained in a cyclotomic number
field. In the function field case, the maximum abelian extension
of $k$ consists of three components: one cyclotomic, one of constants
and one, also cyclotomic, where the infinite prime is totally and 
wildly ramified and it is the only
ramified prime. In the number field case, the ``$p$--components'' can be found
explicitly for $p\geq 3$ depending only on their degree over
${\ma Q}$. The case $p=2$ does not depend only on its degree
over ${\ma Q}$ since, for $n\geq 3$, the
cyclotomic field $\cic 2n$ is not cyclic. We give a criterion to
describe the $2$--component of $\ge K$.
Finally, we present some results on the behavior of the genus
field of a composition. For number fields, a similar result was
obtained by M. Bhaskaran in \cite{Bh79} and by X. Zhang \cite{Xi86}.

\section{Preliminaries and notations}\label{S2}

We denote by $k=\F(T)$ the global rational function
filed with field of constants the finite field of $q$ elements $\F$.
Let $R_T=\F[T]$ be the ring of polynomials, that is, the
ring of integers of $k$ with respect to the pole of $T$, the
infinite prime $\p$. Let $R_T^+:=\{P\in R_T\mid P \text{\ is
monic and irreducible}\}$. The elements of $R_T^+$ are the
{\em finite primes} of $k$ and $\p$ is the {\em infinite prime} of $k$.
For $N\in R_T$, $\Lambda_N$ denotes the $N$--th torsion of 
the Carlitz module. A finite extension $F/k$ will be called 
{\em cyclotomic} if there exists $N\in R_T$ such that
$k\subseteq F\subseteq \cicl N{}$.

Given a cyclotomic function field $E$, the group of Dirichlet characters
$X$ corresponding to $E$ is the group $X$ such that 
$X\subseteq\widehat{(R_T/\langle N\rangle)^*}\cong 
\widehat{\Gal(\cicl N{}/k)}=\Hom\big((R_T/\langle N\rangle)^*,
{\ma C}^*\big)$
and $E=\cicl N{}^{H}$ where $H=\cap_{\chi\in X}\ker \chi$.
For the basic results on Dirichlet characters, we refer to
\cite[\S 12.6]{Vil2006}.

For a group of Dirichlet characters $X$, let 
$Y=\prod_{P\in R_T} X_P$ where $X_P=\{\chi_P\mid 
\chi\in X\}$ and $\chi_P$ is the $P$--th component of $\chi$:
$\chi=\prod_{P\in R_T^+} \chi_P$. If $E$ is the field corresponding
to $X$, we define $\ge E$ as the field corresponding
to $Y$. We have that $\ge E$ is the maximum 
unramified extension at the finite
primes of $E$ contained in a cyclotomic function field.
The infinite prime $\p$ might be ramified in $\ge E/k$
(see \cite{MaRzVi2013}).

Let $L_n=\cicl {1/T^{n+1}}{}^{\f}$, $n\in{\ma N}\cup\{0\}$
where $\*{\F}\subseteq\Big(R_{1/T}/\langle 1/T^{n+1}\rangle
\Big)^*$, is isomorphic to the inertia group of 
the prime corresponding to $T$ in $k\big(\Lambda_{
1/T^{n+1}}\big)/k$. The prime
$\p$ is the only ramified prime in $L_n/k$ and it is totally and
wildly ramified. For $m\in {\ma N}$, and for any finite extension
$F/k$, $F_m$ denotes the extension of constants: $F_m=F
{\ma F}_{q^m}$. In particular $k_m={\ma F}_{q^m}(T)$.

Given a finite abelian extension $K/k$, there exist $n\in{\ma N}
\cup\{0\}$, $m\in{\ma N}$ and $N\in R_T$ such that $K\subseteq
L_n\cicl N{}k_m=:{_n\cicl N{}_m}$ (see \cite[Theorem 12.8.31]{Vil2006}).
We define $M:=L_nk_m$. In $M/k$ no finite prime of $k$ is
ramified.

For any extension $E/F$ of global fields and for any place $\pL$
of $E$ and $\pK=\pL\cap F$, the ramification index is denoted by
$e_{E/F}(\pL|\pK)=e(\pL|\pK)$ and the inertia degree is denoted
by $f_{E/F}(\pL|\pK)=f(\pL|\pK)$. When the extension is Galois
we denote $e_{\pK}(E|F)=e_{E/F}(\pL|\pK)$ and $f_{\pK}(E|F)
=f_{E/F}(\pL|\pK)$. In particular for any abelian extension $E/k$,
$e_P(E/k)$ and $f_P(E/k)$ denote the ramification index and the inertia
degree of $P\in R_T^+$ in $E/k$ respectively, and we denote by
$e_{\infty}(E/k)$ and $f_{\infty}(E/k)$ the ramification index and
the inertia degree of $\p$ in $E/k$. The symbol $\ew EF$ denotes
the wild ramification part of the infinite primes in $E/F$.
Similarly, $I_{E/F}(\pL|\pK)$
denotes the inertia group and $D_{E/F}(\pL|\pK)$ the decomposition
group.

For any finite separable extension $K/k$ the {\em finite primes}
of $K$ are the primes over the primes $P$ in $R_T^+$
and the {\em infinite primes} of $K$ are the primes over $\p$.
The {\em Hilbert class field} $K_H$ of $K$ is the maximum 
abelian extension of $K$
unramified at every finite prime of $K$ and where all the infinite
primes of $K$ are fully decomposed. The {\em genus field} $\g K$
of $K/k$ is the maximum extension of $K$ contained in $K_H$
and such that it is the composite $\g K=KF$ where $F/k$ is abelian.
We choose $F$ the maximum possible. In other words, $F$ is
the maximum abelian extension of $k$ contained in $K_H$.

Let $K/k$ be a finite abelian extension. We know that
$\g K=K\g {E^H}$ is the genus field of $K$ where $H$ is the
decomposition group of the infinite primes in
$KE/K$ and $E:=KM\cap \cicl N{}$ (see \cite{BaMoReRzVi2018}).
We also know that $K\g E/\g K$ and $KE/K$ are extensions
of constants.

For a local field $F$ with prime $\pK$, we denote by $F(\pK)
\cong \F$ the residue field of $F$, $U_{\pK}^{(n)}=1+\pK^n$
the $n$--th units of $F$, $n\in{\ma N}\cup\{0\}$.

Let $\pi=\pi_F=\pi_{\pK}$ be a uniformizer element for $\pK$, that is,
$v_{\pK}(\pi)=1$. Then the multiplicative group of $F$ 
satisfies $\*F\cong\langle\pi\rangle\times U_{\pK}
\cong \langle\pi\rangle\times \*\F\times U_{\pK}^{(1)}$ as groups.

\section{Extended genus field of a global function field}\label{S3}

Let $K/k$ be a finite abelian extension. Let $n\in{\ma N}\cup \{0\}$,
$m\in{\ma N}$ and $N\in R_T$ be such that $K\subseteq {_n\cicl N{}_m}$.
Let $E= KM\cap \cicl N{}$.
Define {\em the extended genus field of $K$} as
\[
\ge K:=K\ge E.
\]
Note that $\ge K/K$ is unramified at
the finite primes since $\ge E/E$ is unramified at the finite primes,
so that $K\ge E/KE$ is unramified at the finite primes and
we also know that $KE/K$
is unramified at the finite primes (\cite{BaMoReRzVi2018}).

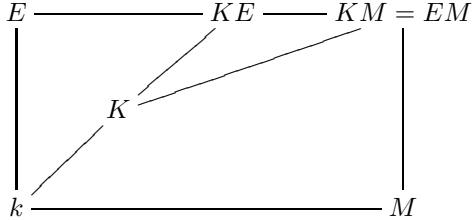
\[
\xymatrix{
E\ar@{-}[rr]\ar@{-}[dd]&&KE\ar@{-}[dl]\ar@{-}[r]&KM=EM
\ar@{-}[dd]\ar@{-}[dll]\\
&K\ar@{-}[dl]\\ k\ar@{-}[rrr]&&&M
}
\]

Now $KM=EM/E$ is ramified at most at the infinite prime $\p$ and
the inertia of $\p$ in the extension $EM=KM/E$ 
is contained in $M$. Hence $EM/E$ is unramified
at the finite primes. The same holds for $KM/K$ and we have $K\subseteq
KE\subseteq KM=EM$. In short, $\ge K/K$ is unramified at the
finite primes. We also have that $\ge K/K$ is tamely ramified at $\p$
since $\ge E/k$ is tamely ramified at $\p$ so that $K\ge E/K$ is tamely
ramified at $\p$ and $\ge K=K\ge E$.

We also have $[\ge E:\g {E^H}]|q-1$ since $e_{\infty}(\ge E|E)|q-1$
where in general, for a finite abelian extension $L/F$,
$e_{\infty}(L/F)$ denotes the ramification index of
the infinite primes of $F$ in $L$, and $H\subseteq I_{\infty}(\ge E|k)$, where
in general $I_{\infty}(L|F)$ denotes the inertia group of
the infinite primes in the Galois 
extension $L/F$. In other words, the infinite primes of 
$\g {E^H}$ are fully ramified in the extension
$\ge E/\g {E^H}$. Thus we have
\[
[\ge E:\g{E^H}]=e_{\infty}(\ge E|\g{E^H})|e_{\infty}(\cicl N{}|k)=q-1.
\]

Therefore we have that $\ge K=K\ge E/K\g {E^H}=\g K$ is unramified
at the finite primes, the infinite primes are 
tamely ramified, and $[\ge K:\g K]|q-1$.

Now let $K/k$ be a finite and separable extension. We define $\ge K$
as $K\ge {F}$ where $\g K=KF$, that is,  
$F$ is the maximum abelian extension of $k$ contained
in the Hilbert class field $K_H$ of $K$
(see \cite{BaMoReRzVi2018}). Note that $\g F=F$.

Note that $[\ge K:\g K]|[\ge {F}:\g{F}]|q-1$ and the only possible 
ramified primes in $\ge K/\g K$ are the infinite primes
and they are tamely ramified.

\begin{definition}\label{D2.1}
For a finite separable extension $K/k$, we define the {\em extended
genus field of $K$} as $\ge K=K\ge F=KL$ where $L=\ge F$. We stress that 
we choose $F$ to be the maximum abelian extension of $k$ such
that $\g K=KF$.
\end{definition}

\begin{remark}\label{R2.2}{\rm{
For a finite prime $P\in R_T^+$, the tame part of the ramification of $P$ in
$\ge K/k$ can be obtained in the following way. Let $d_P=\deg P$
and let $e_P({L}|k)=e_P^{(0)}e_P^{(w)}=
e^{\text{tame}}(P)e_P^{(w)}$ where $\gcd
(p,e_P^{(0)})=1$ and $e_P^{(w)}=p^{\alpha_P}$ for some
integer $\alpha_P\geq 0$.
Since ${L}/k$ is abelian, we have $e^{(0)}|q^{d_P}-1$ (see
\cite[Proposici\'on 10.4.8]{RzeVil2017}).

Consider the extension $k_P^{(0)}/k$ where $P$ is the only finite prime ramified,
$k_P^{(0)}\subseteq {L}$ and $e_P^{(0)}|[k_P^{(0)}:k]$.
Note that $[\cicl P{}:k]=q^{d_P}-1$.
Then $Kk_P^{(0)}\subseteq \ge K$ and $Kk_P^{(0)}/K$ is 
unramfied at the finite primes. Thus, by Abhyankar Lemma, 
\begin{gather*}
e_P(K|k)=e_P(Kk_P^{(0)}|k)=\lcm[e_P(K|k), e_P(k_P^{(0)}|k)]=
\lcm[e_P(K|k),e_P^{(0)}]
\intertext{Therefore $e_P^{(0)}|e_P(K|k)$.
Since $e_P^{(0)}$ is the maximum with this property, it follows
that}
e^{\text{tame}}(P)=e_P^{(0)}=\gcd(q^{d_P}-1,e_P(K|k)).
\end{gather*}
}}
\end{remark}

We now obtain $\ge K=K L$ where $L$ satisfies 
$\ge {L}=L$, $L/k$
abelian and $L$ is the maximum with respect to this property.
Let $L\subseteq {_n\cicl N{}_m}$. If necessary, we may assume
$n,m,N$ are minimum, where $m\in {\ma N}$ is the conductor of constants
(see \cite{BaMoReRzVi2018}), $N\in R_T$ and $n\in{\ma N}\cup\{0\}$.
In this situation we define {\em the conductor of $L$} as $(m,N,n)$.

Let $E:=LM\cap \cicl N{}$. Then $EM=LM$ and
$\ge {L}=L=\ge E L$ so that $\ge E\subseteq L$ and $E=
\ge E$. In fact, $\ge E\subseteq L\subseteq L M=EM$, hence
$\ge E M\subseteq EM$ and from the Galois correspondence, $\ge E
\subseteq E$. Thus $\ge E=E$.

\[
\xymatrix{
E=\ge E\ar@{-}[r]\ar@{-}[d]&L\ar@{-}[r]\ar@{-}[dl]&L M=EM\ar@{-}[d]\\
k\ar@{-}[r]&S\ar@{-}[r]\ar@{-}[u]&M
}
\]

Let $S:=L\cap M$, $S\subseteq M=L_nk_m$. 

Let $X=Y=\prod_{P\in R_T^+} X_P$ be the group of Dirichlet characters
associated to $\ge E=E$. Then if $E_P$ is the field associated to $X_P$,
with $P\in R_T^+$, $E=\prod_{P\in R_T^+} E_P$ where $E_P=k$ for
almost all $P$ and if $P_1,\ldots, P_r$ are the finite primes ramified in 
$E/k$, $X_{P_i}\neq \{1\}$, $E_{P_i}\neq k$, $E_{P_i}\cap \prod_{j\neq i}
E_{P_j}=k$, $1\leq i\leq r$ and $E=E_{P_1}\cdots E_{P_r}$, $\widehat{
\Gal(E/k)}\cong X=Y=\prod_{P\in R_T^+} X_P=\prod_{P\in R_T^+}
\widehat{\Gal(E_P/k)}
\cong \prod_{i=1}^r\widehat{\Gal(E_{P_i}/k)}$. Thus
\[
\Gal(E/k)\cong \prod_{i=1}^r\Gal(E_{P_i}/k).
\]

For any nonempty finite subset ${\mc A}\subseteq R_T^+$, we 
define $E_{\mc A}:=\prod_{P\in{\mc A}}E_P$.
We may consider $E_P$ as the ``$P$--th primary component''
of $E$.
\[
\xymatrix{
E=\prod_{P\in R_T^+}E_P=\ge E\ar@{-}[rr]\ar@{-}[d]&& L=ES\ar@{-}[rr]
\ar@{-}[d]&& L M=EM\ar@{-}[d]\\
E_P\ar@{-}[d]\ar@{-}[rr]&&L_P=E_PS \ar@{-}[d]\ar@{-}[rr]&
&L_PM=E_PM\ar@{-}[d]\\
k\ar@{-}[rr]&&S=L\cap M\ar@{-}[rr]&&M
}
\]

We define $L_P:=E_PM\cap L$. We have that $E_P\subseteq E
\subseteq L$ and $E_P\subseteq E_PM$. Therefore $E_P\subseteq
L_P$. From the Galois correspondence,
we have 
\[
L_P=E_PS. 
\]
For any nonempty finite subset ${\mc A}
\subseteq R_T^+$, we let $L_{\mc A}:=E_{\mc A}M\cap \cicl N{}$.
From the Galois correspondence we obtain $L_{\mc A}=E_{\mc A}
S$ and in particular
\[
L_{\mc A}=\Big(\prod_{P\in{\mc A}}E_P\Big)S=\prod_{P\in{\mc A}}
(E_P S)=\prod_{P\in{\mc A}} L_P.
\]

We have

\begin{proposition}\label{P2.3}
For any $A,B\in R_T\setminus\{0\}$, let $L_A:=
E_AM\cap L$, where $E_A:=\prod\limits_{\substack{P|A\\ P\in
R_T^+}}E_P$, that is, $E_A=E_{\mc A}$ and $L_A=L_{
\mc A}$, where ${\mc A}=\{P\in R_T^+\mid P|A\}$. Then we have 
\[
L_{AB}=L_AL_B.
\]

Furthermore, if $\gcd(A,B)=1$ we have $L_A\cap L_B=S
=L\cap M$.
\end{proposition}

\begin{proof} It remains to consider the case $\gcd(A,B)=1$.
We have $E_A=\prod_{P|A} E_P$, $E_B=\prod_{P|B}E_P$
and $\{P\in R_T^+\mid P|A\}\cap
 \{P\in R_T^+\mid P|B\}=\emptyset$. Therefore $E_A\cap E_B
=k$ and $kL\cap M=L\cap M=S$. The result follows from the Galois
correspondence.
\end{proof}

Now, for $P\in R_T^+$, $L_P=E_PM\cap L\supseteq M\cap 
L=S$ and $L_P\neq S\iff P\in\{P_1,\ldots,P_r\}$. In fact,
$L_P=E_PM\cap L\neq S\iff E_PM\neq M\iff E_P\neq k\iff
P\in\{P_1,\ldots,P_r\}$.

Finally, $E=\prod_{P\in R_T^+}E_P=\prod_{i=1}^rE_{P_i}$.
Therefore, since $EM=LM$, in particular $L\subseteq EM$. We have
\begin{gather*}
\begin{align*}
L&=EM\cap L=\big(\prod_{i=1}^rE_{P_i}\big) M\cap L=
\prod_{i=1}^r(E_{P_i}M\cap L)=\prod_{i=1}^rL_{P_i}\\
&=\prod_{i=1}^rL_{P_i}\cdot
\prod_{P\notin \{P_1,\ldots,P_r\}}S=\prod_{P\in R_T^+}L_P.
\end{align*}
\intertext{Thus}
L=\prod_{i=1}^rL_{P_i}=\prod_{P\in R_T^+}L_P.
\end{gather*}

We have proved

\begin{theorem}\label{T2.4}
For $A\in R_T$, we define
$L_A=E_AM\cap L$. Let $S=L\cap M$. We have
\las
\item For all $A, B\in R_T$, $L_{AB}=L_AL_B$.

\item $L_A\supseteq S$ for all $A\in R_T$ and $L_A=S\iff
P_i\nmid A$ for all $1\leq i\leq r$.

\item $L_A\cap L_B=S$ for all $A,B\in R_T$ such that
$\gcd(A,B)=1$.

\item $L=\prod_{P\in R_T^+}L_P=\prod_{i=1}^r L_{P_i}$.
$\fin$
\end{list}
\end{theorem}

In order to compute $L$ we need to know $S$, that is, the behavior
of $\p$, and also each $E_P$ for $P\in R_T^+$. First, we have that if
$P\in R_T^+$ is unramified in $K/k$, then $P$ is unramified in $E/k$
and therefore in $L/k$. Indeed, if $P$ were ramified in $L/k$, then
we would have 
\[
e_P(KL|K) =e_P(KL|K)e_P(K|k)=e_P(KL|k)=
e_P(KL|L)e_P(L|k)>1
\]
so that $e_P(KL|K)>1$ contrary to the definition of $L$.
\[
\xymatrix{
K\ar@{-}[r]\ar@{-}[d]&KL\ar@{-}[d]\\k\ar@{-}[r]&L}
\]

Thus, it suffices to know $E_{P_i}$, $1\leq i\leq r$ where
$P_1,\ldots,P_r$ are the finite primes ramified in $K/k$ and
therefore these are the only possible finite primes ramified in
$E/k$ and in $L/k$.
Now, in $E_P/k$ the only finite prime ramified is $P$
and $\p$ is tamely ramified. Note that the tame ramification
index of $\p$ in $E/k$ and in $L/k$ is the same. This is
a consequence of $L=ES$.

In general
we consider an arbitrary global function field $F$. Let $J_F$ be the 
id\`ele group of $F$ and let $C_F=J_F/\*F$ be the id\`ele class group
of $F$. To find $E_P$ for $P\in\{P_1,\ldots,P_r\}$,
we must find the id\`ele subgroup of $J_k$ corresponding to
$E_P$.  Now, since $E_P$ is cyclotomic and $P$ is the only
finite prime ramified, there exists $t\in{\ma N}$ such that $E_P
\subseteq \cicl Pt$. Therefore, the id\`ele group corresponding to
$E_P$ contains the id\`ele group corresponding to $\cicl Pt$.

\begin{theorem}\label{T3.1N} Let $N\in R_T$, $N=P_1^{\alpha_1}\cdots
P_r^{\alpha_r}$ with $P_1,\ldots,P_r\in R_T^+$ distinct. Set 
$R_T':=R_T^+\setminus\{P_1,\ldots,P_r\}$. Then,
the id\`ele group corresponding to $\cicl N{}$ is 
\[
{\mc X}_N=\prod_{i=1}^r U_{P_i}^{(\alpha_i)}\times \prod_{P\in R_T'} U_P
\times [(\pi)\times U_{\infty}^{(1)}],
\]
where $\pi=1/T$ is a uniformizer for $\p$.
\end{theorem}

\begin{proof}
Let $U':=\prod_{Q\in R_T^+}U_Q\times [(\pi)\times U_{\infty}^{(1)}]$.
We will give an epimorphism 
\[
\psi_N\colon U'\lra
\Gal(\cicl N{}/k)=:G_N
\] such that $\ker \psi_N={\mc X}_N$
and hence, $U'/{\mc X}_N\cong G_N$.

Let $\vec\xi\in U'$. Then 
$\xi_{P_i}\in U_{P_i}=\{\sum_{j=0}^{\infty} a_jP_i^j\mid a_j\in R_T/
\langle P_i\rangle\}$, $1\leq i\leq r$. Since $k$ is dense in 
the local field $k_{P_i}$,
there exists $Q_i\in R_T$ such that
$Q_i\equiv \xi_{P_i}\bmod P_i^{\alpha_i}$. By the Chinese
Residue Theorem, we have that there exists
$C\in R_T$ such that 
$C\equiv Q_i\bmod P_i^{\alpha_i}$, $1\leq i\leq r$
and so $C\equiv \xi_{P_i}\bmod P_i^{\alpha_i}$, $1\leq i\leq r$

Now, if $C_1\in R_T$ satisfies $C_1\equiv \xi_{P_i}
\bmod P_i^{\alpha_i}$, $1\leq i\leq r$, we have that $P_i^{\alpha_i}
|C-C_1$ for $1\leq i\leq r$. It follows that $N|C-C_1$ and thus
$C\in R_T$ is unique modulo $N$. On the other hand,
$v_{P_i}(\xi_{P_i})=0$, so that $P_i\nmid \xi_{P_i}$ 
and so we obtain that $\gcd(C,N)=1$. In this way we have that
$C\bmod N$ defines an element of $G_N=\Gal(\cicl N{}/k)$.

Given $\sigma\in G_N$, there exists $C\in R_T$ such that $\sigma
\lambda_N=\lambda_N^C$ where $\lambda_N$
is a generator de $\Lambda_N$. Let $\vec\xi\in U'$
with $\xi_{P_i}=C$, $1\leq i\leq r$ and $\xi_P=1=\xi_{\infty}$
for all $P\in R_T'$. Therefore $\vec\xi\mapsto C\bmod
N$ and $\psi_N$ es onto. Finally, $\ker\psi_N
=\{\vec\xi\in U'\mid\xi_{P_i}\equiv 1 \bmod P_i^{\alpha_i},
1\leq i\leq r\}={\mc X}_N$. So we have that $\psi_N$ is
an epimorphism and $\ker \psi_N={\mc X}_N$.

We will show that $U'/{\mc X}_N\cong J_k/{\mc X}_N\*k$.
We have the composition
\[
\xymatrix{
U'\ar@{^{(}->}[r]\ar@/_1pc/@{>}_{\mu}[rr]&
J_k\ar@{>>}[r]&J_k/{\mc X}_N \*k,
}
\]
with $\im \mu=U'{\mc X}_N\*k/{\mc X}_N\*k$ and $\ker \mu =
U'\cap {\mc X}_N\*k$.

Now, ${\mc X}_N\subseteq U'$ so that ${\mc X}_N
\subseteq U'\cap {\mc X}_N\*k$. Conversely, if $\vec\xi
\in U'\cap {\mc X}_N\*k$, the components of
$\vec\xi$ are given as
\begin{align*}
\xi_P&=a\cdot \beta_P, \quad P\in R_T, \quad
\vec\beta \in {\mc X}_N, \quad a\in \*k,\\
\xi_{\infty}&=a\cdot\beta_{\infty}, \quad
\beta_{\infty}\in (\pi)\times U_{\infty}^{(1)}.
\end{align*}

Since $\xi_P,\beta_P\in U_P$ we have $v_P(\xi_P)=
v_P(\beta_P)=0$ for all $P\in R_T$. It follows that 
$v_P(a)=0$ for all $P\in R_T$. Furthermore, since
$\deg a=0$ we have $v_{\infty}(a)=0$ and so $a\in\*\F$.

Now $\xi_{\infty}, \beta_{\infty}
\in(\pi)\times U_{\infty}^{(1)}=\ker \phi_{\infty}$, where
$\phi_{\infty}\colon \*{k_{\infty}}\lra \*{\F}$ is the sign function
of $\*{k_{\infty}}$ defined as $\phi_{\infty}(\lambda\pi^n u)=
\lambda$ where $\lambda\in \*{\F}$, $n\in{\ma N}$ and
$u\in U_{\infty}^{(1)}$.
Thus $1=\phi_{\infty}(\xi_{\infty})=
\phi_{\infty} (a) \phi_{\infty} (\beta_{\infty})=
\phi_{\infty} (a)$ and so
$a=1$. It follows that $\vec\xi\in{\mc X}_N$.
Therefore $\ker \mu={\mc X}_N$ and we obtain a monomorphism
$U'/{\mc X}_N\stackrel{\theta}
{\hooklongrightarrow} J_k/{\mc X}_N \*k$.

It remains to prove that $\theta$ is surjective. So, we 
must prove that $J_k=U'{\mc X}_N \*k
=U'\*k$. We have that $U'$ corresponds
to the maximum unramified extension at every finite prime.
Let $L/k$ be this extension. Since $U_{\infty}^{
(1)}$ corresponds to the first ramification group, and in
this way it corresponds to the wild ramification of $\p$,
it follows that in $L/k$ there is at most a ramified prime ($\p$),
being tamely ramified and of degree $1$.
From \cite[Proposici\'on 10.4.11]{RzeVil2017}, we obtain that 
$L/k$ is an extension of constants.

Finally, since $1=\min\{n\in{\ma N}\mid \deg\vec
\alpha=n, \vec\alpha\in U'\}$, the field of constants of $L$ is
$\F$ (see \cite[Teorema 17.8.6]{RzeVil2017})
and therefore $L=k$. It follows that $C_k
\cong U'$, that is, $J_k/\*k\cong U'$ and thus
$J_k=U' \*k$.
\end{proof}

\begin{corollary}\label{C3.2N} With the above notations, we have that
for a cyclotomic field $k\subseteq F\subseteq \cicl N{}$, the
id\`ele group corresponding to $F$ is of the form $R_F\times 
\prod_{Q\in R_T'}U_Q\times [(\pi)\times U_{\infty}^{(1)}]$ with 
$R_F$ a group satisfying $\prod_{i=1}^r U_{P_i}^{(\alpha_i)}\subseteq
R_F\subseteq \prod_{i=1}^r U_{P_i}$.
\end{corollary}

\begin{proof}
Let $\Delta$ be the id\`ele group corresponding to $F$. Thus
$\Delta\supseteq {\mc X}_N$. Now 
\[
\frac{
\prod_{i=1}^r U_{P_i}}{\prod_{i=1}^r U_{P_i}^{(\alpha_i)}}\cong
\big(R_T/\langle N\rangle\big)^*\cong \Gal(\cicl N{}/k).
\]
Therefore
$\Gal(\cicl N{}/F)\cong\frac{\Theta}{\prod_{i=1}^r U_{P_i}^{(\alpha_i)}}
<\Gal(\cicl N{}/k)$
for a group $\Theta\subseteq \prod_{i=1}^r U_{P_i}$. The group
$\Theta$ corresponds to $R_F$. The result follows.
\end{proof}

\begin{corollary}\label{C3.3N} Let $P\in R_T^+$. Then
the id\`ele group corresponding to $E_P$ is of the form
\[
\Delta_P=H_P\times \prod_{\substack{Q\neq P\\Q\in R_T^+}}
U_Q\times [(\pi)\times U_{\infty}^{(1)}],
\]
where $U_P^{(t)}\subseteq H_P\subseteq U_P$ for some $t\in
{\ma N}$.
\end{corollary}

\begin{proof} Since $E_P$ is cyclotomic and the only
finite prime ramified is $P$, there exists $t\in {\ma N}$ such that
$E_P\subseteq \cicl Pt$. The result follows from
Corollary \ref{C3.2N} 
\end{proof}

For each $P\in R_T^+$, $k_P$ denotes
the completion of $k$ at $P$ and $k_{\infty}$ denotes the
completion of $k$ at $\p$.
We recall the following result of class field theory.

\begin{theorem}\label{T2.5}
Let $F$ be a global field and let $R/F$ be the class field extension
corresponding to $H$, that is, $H$ is the open subgroup of $C_F$
such that $H=\N_{R/F} C_R$ and $\Gal(R/F)\cong C_F/H$. Let
$E/F$ be a finite separable extension. Then $ER/E$ is the class
field extension corresponding to the subgroup $\N_{E/F}^{-1}(H)$ of $C_E$.
\[
\xymatrix{
E\ar@{-}[rr]^{\N_{E/F}^{-1}(H)}\ar@{-}[d]&&ER\ar@{-}[d]\\
F\ar@{-}[rr]_H&&R.
}
\]
\end{theorem}

\begin{proof}
We have that if $E/F$ and $E'/F'$ are two finite abelian extensions
of global fields with $F\subseteq F'$ and $E\subseteq E'$ of global fields, and
if $\psi_{E/F}$ denotes the Artin map of the extension $E/F$
then we have the following commutative diagram
\[
\xymatrix{
C_{F'}\ar@{->}[rr]^{\psi_{E'/F'}}\ar@{->}[d]_{\N_{F'/F}}&
&\Gal(E'/F')\ar@{->}[d]^{\rest}\\
C_F\ar@{->}[rr]_{\psi_{E/F}}&&\Gal(E/F)
}
\]
where $\rest$ denotes the restriction map (see \cite[Proposici\'on
17.6.39]{RzeVil2017}).

We apply this result to our situation, that is, we have the commutative
diagram
\[
\xymatrix{
C_{E}\ar@{->}[rr]^{\psi_{ER/E}}\ar@{->}[d]_{\N_{E/F}}&
&\Gal(ER/E)\ar@{->}[d]^{\rest}\\
C_F\ar@{->}[rr]_{\psi_{R/F}}&&\Gal(R/F)
}
\]

Let $\psi_{ER/E}\colon
C_E\lra \Gal(ER/E)$ be the Artin map. The norm group corresponding
to $ER/E$ is $\ker\psi_{ER/E}$, that is, $C_E/\ker\psi_{ER/E}\cong
\Gal(ER/E)$. Now the restriction map is injective and we have
\begin{gather*}
\rest\circ\psi_{ER/E}=\psi_{R/F}\circ \N_{E/F}. 
\intertext{Therefore}
\vec x\in\ker\psi_{ER/E}\iff \psi_{ER/E}(\vec x)=1\iff \\
\iff \rest\circ 
\psi_{ER/E}(\vec x)
=1=\psi_{R/F}\circ \N_{E/F}(\vec x)\iff \\
\iff \N_{E/F}(\vec x)\in\ker\psi_{R/F}=H\iff
\vec x\in \N_{E/F}^{-1}(H).
\end{gather*}
\end{proof}

We apply Theorem \ref{T2.5} to the diagram
\[
\xymatrix{
K\ar@{-}[rr]\ar@{-}[d]&&KE_P\ar@{-}[d]\\ k\ar@{-}[rr]&&E_P}
\]
that is, $KE_P$ is the class field of $\N_{K/k}^{-1}(\Delta_P)$.
Since $E_P$ is maximum in the sense that $P$ is the only
finite prime ramified in $E_P/k$ and $KE_P/K$ unramified at
every finite prime, we have that $\Delta_P$ satisfies
\begin{gather*}
N_{K/k}^{-1}(\Delta_P)\subseteq \prod_{Q\in R_T^+}
\prod_{\pK|Q}U_{\pK}\times \prod_{\pL_{\infty}|\p}
\*{K_{\pL_{\infty}}}\subseteq J_K,
\intertext{or, equivalently,}
\Delta_P\subseteq \N_{K/k}\big(\prod_{Q\in R_T^+}
\prod_{\pK|Q}U_{\pK}\times\prod_{\pL_{
\infty}|\p}\*{K_{\pL_{\infty}}}\big).
\end{gather*}

Let $\vec \alpha\in \prod_{Q\in R_T^+}
\prod_{\pK|Q}U_{\pK}\times\prod_{\pL_{
\infty}|\p}\*{K_{\pL_{\infty}}}$, $\vec \alpha=(\alpha_{\pK})_{\pK}$.
Then 
\[
\N_{K/k}\vec \alpha=\prod_{Q\in R_T^+}\big(
\prod_{\pK|Q}\N_{K_{\pK}/k_Q}\alpha_{\pK}\big)\cdot
\prod_{\pL_{\infty}|\p}\N_{K_{\pL_{
\infty}/k_{\infty}}}\alpha_{\pL_{\infty}}.
\]

For $Q\neq P$, $Q$ is unramified in $L_P/k$, therefore,
for ${\eu Q}|Q$, $K_{\eu Q}/k_Q$ is unramified and in particular
it is a cyclic extension. Then $\N_{K_{\eu Q}/k_Q} U_{\eu Q}=
U_Q$ (see \cite[Teorema 17.2.17]{RzeVil2017}).

For $Q=P$, we have 
\[
\prod_{\pK|P}\N_{K_{\pK}/k_P}\alpha_{\pK}=\prod_{j=1}^{m_P}
\N_{K_{\pK_j/k_P}} \alpha_{\pK_j}
\]
where $\con_{k/K} P=\pK_1^{e_1}\cdots \pK_{m_P}^{e_{m_P}}$.

It follows that $\prod_{j=1}^{m_P}
\N_{K_{\pK_j/k_P}} \alpha_{\pK_j}\in H_P$. In other words, if
\[
S_j:=\N_{K_{\pK_j/k_P}} U_{\pK_j}\times \prod_{\substack{Q\in
R_T^+\\ Q\neq P}}U_Q \times [(\pi)\times U_{\infty}^{(1)}] \subseteq 
U_P\times  \prod_{\substack{Q\in
R_T^+\\ Q\neq P}}U_Q \times [(\pi)\times U_{\infty}^{(1)}],
\]
we have
\[
\Delta_P=\prod_{j=1}^{m_P}S_j\quad\text{and}\quad
H_P=\prod_{j=1}^{m_P} \N_{K_{\pK_j/k_P}} U_{\pK_j}.
\]

Now, if $S_j$ is the norm group of the field $R_j\subseteq {_n
\cicl P{c_j}_m}$ for some $n\in{\ma N}\cup\{0\}, m\in{\ma N}$
and $c_j\in{\ma N}$, then $\prod_{j=1}^{m_P} S_j$ is the norm
group of $\cap_{j=1}^{m_P} R_j$. 

It follows that $[C_k:k^*S_j]=[U_P:\N_{K_{\pK_j}/k_P} U_{\pK_j}]$
and $\Gal(R_j/k)\cong C_k/\*k S_j$. Therefore $[R_j:k]=[C_k:k^*S_j]=
[U_P:\N_{K_{\pK_j}/k_P} U_{\pK_j}]$. Finally, we have
\begin{gather*}
E_P=\bigcap_{j=1}^{m_P} R_j,\quad [E_P:k]=\Big[\bigcap_{j=1}^{m_P}
R_j:k\Big]=\Big[U_P:\prod_{j=1}^{m_P} 
\N_{K_{\pK_j/k_P}} U_{\pK_j}\Big].
\end{gather*}

We have proved our main result.

\begin{theorem}\label{T2.6}
Let $K/k$ be a finite and separable extension, where $k=\F(T)$.
With the notations as above, let $\ge K=KL$. Then $L=
\prod_{P\in R_T^+}L_P$ where $L_P=E_PS$, $S=L\cap M$
and $k\subseteq E_P\subseteq \cicl P{c_P}$ corresponds to
$\prod_{j=1}^{m_P} 
\N_{K_{\pK_j/k_P}} U_{\pK_j}$. In particular
\[
[E_P:k]=\Big[U_P:\prod_{j=1}^{m_P} 
\N_{K_{\pK_j/k_P}} U_{\pK_j}\Big],
\]
where $\con_{k/K} P=\pK_1^{e_1}\cdots \pK_{m_P}^{e_{m_P}}$.

The tamely ramified part of $L_P/K$ is given by
\begin{gather*}
e^{\rm{tame}}(P)=\gcd(e_1,\ldots,e_{m_P},q^{d_P}-1),
\end{gather*}
with $d_P=\deg_k P$.
$\fin$
\end{theorem}

\subsection{The field $S$}\label{S3.1}

To study $S$, recall that for a finite extension $K/k$, the genus
field is $\g K=KF$ where $F/k$ is the maximum abelian extension
contained in the Hilbert class field and the extended genus field is
$\ge K=KL$, where $L$ satisfies $\ge L=L$, $L/k$ is abelian
and $L$ is the maximum with respecto to this property. We have
$\g F=F$, $L=\ge F$ and $\ge L=L$. Let $L\subseteq {_n\cicl N{}_m}$
with $(m, N, n)$ the conductor of $L$. Then $M=L_nk_m$
and $S=L\cap M$.

\begin{proposition}\label{P3.1.1}
We have that $L/F$ is totally ramified at the infinite primes, unramified
at the finite primes and $[L:F]|q-1$. In particular, $L/F$ is tamely 
ramified.
\end{proposition}

\begin{proof}
We have that $F/k$ is abelian. Let $F\subseteq {_n\cicl N{}_m}$ and
$E=FM\cap \cicl N{}$. Then $\g EM=\g FM=FM=EM$
(see \cite{BaMoReRzVi2018}) and therefore $\g E=E$.

Since $e_{\infty}(\ge E|\g E)|q-1$ and $e_{\infty}(M|k)=q^n$, it follows
that $e_{\infty}(\ge EM/\g EM)=e_{\infty}(\ge E|\g E)=[\ge E:\g E]$.
\[
\xymatrix{
\ge E\ar@{-}[rrr]_{e_{\infty}=q^n}
\ar@{-}[d]_{e_{\infty}=d|q-1}&&&\ge EM=\ge FM\ar@{-}[d]^{e_{\infty}=d|q-1}\\
\g E\ar@{-}[rrr]_{e_{\infty}=q^n\phantom{xxxxxxx}}
\ar@{-}[d]&&&\g EM=\g FM=FM=EM\ar@{-}[d]\\
k\ar@{-}[rrr]_{e_{\infty}=q^n}&&&M
}
\]

Hence, $e_{\infty}(\ge F|\g F)=e_{\infty}(\ge F M|\g FM)=e_{\infty}(
\ge EM|\g EM)=[\ge E:\g E]=[\ge E F:\g EF]=[\ge F:\g F]$. So, the 
infinite primes are total and tamely ramified in $L=\ge F/\g F=F$.

On the other hand, $\ge E/\g E$ is unramified at the finite primes,
thus $\ge EF=\ge F=L/F=\g F=\g EF$ is unramified at the finite
primes.
\end{proof}

\begin{proposition}\label{P3.1.2}
We have
\[
\ew {L}k=\ew Fk=\ew Sk=e_{\infty}(S|k).
\]
Furthermore, $S=L\cap M=F\cap M$.
\end{proposition}

\begin{proof}
We have $\ew{L}k=\ew LK=\ew LF\ew Fk=\ew Fk$.

By the definition of $S$, we have $\ew Sk=e_{\infty}(S|k)$ since
$e_{\infty}(S|k)|q^n$. Now, $L=\ge ES$, $\ge E\cap S=k$ and
$\ew {\ge E}k=1$. Therefore 
\[
\ew Lk=\ew {\ge ES}S\ew Sk=\ew Sk
\]
since $\ew {\ge ES}S|\ew {\ge E}k=1$.

We have $F\cap M\subseteq L\cap M=S$ and $F\cap(L\cap M)=F\cap M$.
\[
\xymatrix{
S=L\cap M\ar@{-}[r]\ar@{-}[d]&F(L\cap M)\ar@{-}[d]\ar@{-}[r]&L\\
F\cap M\ar@{-}[r]&F\ar@{-}[ru]
}
\]
It follows that $[L\cap M:F\cap M]=[F(L\cap M):F]|[L:F]|q-1$. We have
that $L/F$
is totally ramified at the infinite primes and therefore $F(L\cap M)/F$ is
also fully ramified at the infinite primes. It follows that $S=L\cap M/F\cap M$
is fully ramified at the infinite primes (see \cite[Corolario 10.4.15]{RzeVil2017}).
Thus, $[S:F\cap M]|q^n$ and $[S:F\cap M]|q-1$ so that 
$[S:F\cap M]=1$ and $F\cap M=L\cap M$.
\end{proof}

\begin{proposition}\label{P3.1.3}
The field of constants of $S$, of $L$ and of $F$ is the same.
\end{proposition}

\begin{proof}
If ${\ma F}_{q^{t_0}}$ is the field of constants of $L$ then ${\ma F}_{q^{t_0}}
\subseteq S=L\cap M$ and since $S\subseteq F\subseteq L$, the
result follows.
\end{proof}

\begin{proposition}\label{P3.1.4}
Let $\con_{k/K}\p={\mc P}_1^{e_1}\cdots {\mc P}_r^{e_r}$ and let $t_i=\deg_K(
{\mc P}_i)$. Then the field of constants of $\g K$ is ${\ma F}_{q^{t_0}}$ where
$t_0=\gcd(t_1,\ldots, t_r)$.
\end{proposition}

\begin{proof}
See \cite{MaRzVi2017}.
\end{proof}

\begin{corollary}\label{C3.1.4'}
The field of constants of $S, L$ and $F$ is ${\ma F}_{q^{t_0}}$.
$\fin$
\end{corollary}

Now we consider a finite abelian extension $J/k$ such that $KJ/K$ is
unramified and the infinite primes decompose fully. Let $\pL|\p$ be a
prime divisor of $KJ$, $\pL\cap K={\mc P}_i$ for some $1\leq i\leq r$ and
$\pL\cap J={\eu Q}$. Taking the completions we have
\[
\xymatrix{
K_{{\mc P}_i}\ar@{=}[r]^{=1}\ar@{-}[d]&(KJ)_{\pL}\ar@{-}[d]\\
k_{\infty}\ar@{-}[r]_{H_i}&J_{\mc Q}
}
\]
Let $H_i:=\N_{J_{\mc Q}/k_{\infty}}J_{\mc Q}^*$, that is,
$H_i$ is the norm group of $J_{\mc Q}$. Therefore, the norm
group corresponding to $(KJ)_{\pL}=K_{\P_i}$ is
$\N_{K_{\P_i}/k_{\infty}}^{-1}(H_i)$ (see Theorem \ref{T2.5}). Hence
$\N_{K_{\P_i}/k_{\infty}}^{-1}(H_i)
=K_{\P_i}^*$. That is, $H_i=\N_{K_{\P_i}/k_{\infty}}(K_{\P_i}^*)$.
The maximum global abelian extension $J/k$ satisfying that $KJ/K$
is unramified and the infinite primes decompose fully, satisfies,
locally at $\infty$, that its norm gorup is
\[
\prod_{i=1}^rH_i=\prod_{i=1}^r \N_{K_{\P_i}/k_{\infty}}(K_{\P_i}^*).
\]

In this way, if $R/k_{\infty}$ is the maximum abelian extension
with $(KR)_{\pL}=K_{\P_i}$ for some $i$. Thus $R$ corresponds
to $\prod_{i=1}^rH_i$, that is, $\Gal(R/k_{\infty})=k_{\infty}^*/
(\prod_{i=1}^r H_i)$ and $[R:k_{\infty}]=[k_{\infty}^*:\prod_{i=1}^r H_i]$.

Let $[R:k_{\infty}]=p^{\alpha}a$ with $\alpha\in{\ma N}\cup\{0\}$
and $p\nmid a$. Since $S$ is the
maximum abelian extension of $k$ such that the only ramified
prime is $\p$, it is fully ramified and $S\subseteq L$, and since
$f_{\infty}(L|S)=1$, it follows that if $\P_{\infty}$ is the only prime
in $S$ dividing $\p$ (recall that the number of primes
in $S$ that lie above $\p$ is $h_{\infty}(S|k)=1$), then
$[S_{\P_{\infty}}:k_{\infty}]=[S:k]=p^{\alpha}$. In particular, the
norm group corresponding to $S_{\infty}=S_{\P_{\infty}}$ in $k_{\infty}$ is the 
group ${\mc S}\supseteq \prod_{i=1}^r H_i$,
which is the minimum such that
$[k_{\infty}^*:{\mc S}]=p^{\alpha}$ is a $p$--group.

The conductor $\p^{n_0}$ of $S_{\infty}$ is such that $n_0$ is the
minimum nonnegative integer such that $U_{\infty}^{(n_0)}
\subseteq {\mc S}$. The conductor of constants $m_0$ of $S$,
that is, $m_0$ is the minimum natural number such that
$S\subseteq k_{m_0}L_{n_0}$ is given as follows 
(see \cite{BaMoReRzVi2018}). Let $t=f_{\infty}(S|k)$, $d^*=
f_{\infty}(R'S|S)$ where $R'=S_{m_0}\cap L_{n_0}$ and
$d^*=e_{\infty}(S|F')$ where $F'=S\cap {_{n_0}\cicl 1{}}=S\cap L_{n_0}$.
Therefore 
\[
m_0=f_{\infty}(S|k)e_{\infty}(S|S\cap L_{n_0}).
\]

\begin{proposition}\label{P3.1.5}
Let $f_{\infty}(S|k)=t$. Then ${\ma F}_{q^t}$ is the field of constants
of $S$. That is, $t=t_0$.
\end{proposition}

\begin{proof}
We have ${\ma F}_{q^{t_0}}(T)=k_{t_0}\subseteq S$. Let $m_0, n_0$
be minimum such that $S\subseteq k_{m_0}L_{n_0}$. Then $S\cap k_{
m_0}=k_{t_0}$.
\[
\xymatrix{
k_{m_0}\ar@{-}[d]\ar@{-}[rrr]&&&k_{m_0}S\ar@{-}[ldd]|!{[llld];[drr]}\hole
\ar@{-}[rr]&&k_{m_0}L_{n_0}\ar@{-}[d]\\
k_{m'}\ar@{-}[rrrrr]\ar@{-}[dd]&&&&&k_{m'}L_{n_0}\ar@{-}[dd]\\
&&S\ar@{-}[dll]\ar@{-}[rrru]\\
k_{t_0}\ar@{-}[rrrrr]&&&&&k_{t_0}L_{n_0}
}
\]
We have $Sk_{t_0}L_{n_0}=SL_{n_0}$, $k_{t_0}L_{n_0}\subseteq
SL_{n_0}\subseteq k_{m_0}L_{n_0}$. Let $SL_{n_0}\cap
k_{m_0}=k_{m'}$. From the Galois correspondence we obtain that
$k_{m'}k_{t_0}L_{n_0}=k_{m'}L_{n_0}=SL_{n_0}\supseteq S$. 
It follows that $m'\geq m_0$. Hence $m'=m_0$ and $SL_{n_0}
=k_{m_0}L_{n_0}$.

Now $e_{\infty}(k_{m_0}L_{n_0}|k_{t_0})=q^n$ and $k_{m_0}
\subseteq k_{m_0}S\subseteq k_{m_0}L_{n_0}$. Then
\begin{gather*}
e_{\infty}(k_{m_0}L_{n_0}|S)=e_{\infty}(k_{m_0}L_{n_0}|k_{m_0}S)=
\frac{q^{n_0}}{[k_{m_0}S:k_{m_0}]}=\frac{q^{n_0}}{[S:S\cap k_{m_0}]}
=\frac{q^{n_0}}{[S:k_{t_0}]}.
\intertext{Thus}
e_{\infty}(S|k_{t_0})=\frac{e_{\infty}(k_{m_0}L_{n_0}|k_{t_0})}{e_{
\infty}(k_{m_0}L_{n_0}|S)}=\frac{q^{n_0}}{q^{n_0}/[S:k_{t_0}]}=
[S:k_{t_0}].
\end{gather*}
It follows that $S/k_{t_0}$ is fully ramified at the infinite prime.
In particular $f_{\infty}(S|k_{t_0})=1$ so that $f_{\infty}(S|k)=
f_{\infty}(S|k_{t_0})f_{\infty}(k_{t_0}|k)=f_{\infty}(k_{t_0}|k)=t_0$.
\end{proof}

We collect the above discussion in the following theorem.

\begin{theorem}\label{T3.1.6}
Let $S=L\cap M$. Let  $\con_{k/K}\p=\P_1^{e_1}\cdots
\P_r^{e_r}$ and let $t_i=\deg_K(\P_i)$, $1\leq i\leq r$. Then the 
field of constants of $S$ is ${\ma F}_{q^{t_0}}$.

 Let $n_0$ be the
minimum nonnegative integer with $U_{\infty}^{(n_0)}\subseteq
{\mc S}$ where ${\mc S}\supseteq \prod_{i=1}^rH_i=\prod_{i=1}^r
\N_{K_{\P_i}/k_{\infty}}(K_{\P_i}^*)$ and ${\mc S}$ is the minimum
such that $[k_{\infty}^*:{\mc S}]=p^{\alpha}$ is a $p$--group. Then
the conductor of constants of $S$ is $m_0=f_{\infty}(S|k) e_{\infty}(
S|S\cap L_{n_0})=t_0e_{\infty}(S|S\cap L_{n_0})$ and ${\mc S}$
is the local norm group corresponding to $S$. In particular
${\ma F}_{q^{t_0}}\subseteq S\subseteq k_{m_0}L_{n_0}$.
$\fin$
\end{theorem}

\section{Number fields}\label{S4}

The results of Section \ref{S3} can be developed in the number
field case. In fact, for a number field, the extended genus field
is more transparent than in the function field case.

\begin{definition}\label{D3.1}
Let $K$ be an arbitrary number field, that is, a finite extension
of the rational field ${\ma Q}$. Let $K_{H^+}$ be the extended
or narrow Hilbert class field of $K$, that is, $K_{H^+}$ is
the maximum abelian extension of $K$ unramified at every
finite prime of $K$. We define the {\em extended genus
field $\ge K$} of $K$ as the maximum extension of $K$
contained in $K_{H^+}$ such that it is of the form $KL$
with $L/{\ma Q}$ abelian.

Equivalently, if $L$ is the maximum abelian extension of ${\ma Q}$
contained in $K_{H^+}$, the extended genus field of $K$ is
$\ge K=KL$.
\end{definition}

Again, we stress that we choose $L$ maximum.

As in the function field case, we have

\begin{proposition}\label{P4.1N}
Let $K/{\ma Q}$ be a finite abelian extension and let $X$ be 
the group of Dirichlet characters corresponding to $K$. Then
$Y:=\prod_{p \text{\ prime}}X_p$ is the group of Dirichlet
characters corresponding to $\ge K$. 
$\fin$
\end{proposition}

In particular, if $K/{\ma Q}$ is any finite extension and $\ge
K=KL$, then $L=\ge L$.
We want to describe $\ge K$ for a general number field $K$.
Let $K/{\ma Q}$ be a finite extension.
Let $p$ be a prime in ${\ma Q}$ and let
\[
\con_{{\ma Q}/K} p ={\pK}_1^{e_1}\cdots {\pK}_r^{e_r},
\]
that is, $e_i=e_{K|{\ma Q}}({\pK}_i|p)$, $1\leq i\leq r$.
Let $K_{{\pK}_1},\ldots, K_{{\pK}_r}$ be the completions
of $K$ at the primes above $p$. Let $\ge K=KL$ with
$L/{\ma Q}$ the maximum abelian extension 
such that $K\subseteq \ge K\subseteq K_{H^+}$.
\[
\xymatrix{
K\ar@{-}[r]\ar@{-}[d]&\ge K=KL\ar@{-}[r]\ar@{-}[d]&K_{H^+}\\
{\ma Q}\ar@{-}[r]&L}
\]

Since $L=\ge L$, we let $L_p$ be the field corresponding to $X_p$.
We have that $L=\prod_{p \text{\ prime}}L_p$ and $L_p\cap
L_q={\ma Q}$ for any primes $p,q$ such that $p\neq q$.
We have that $L_p$ is the maximum abelian extension of ${\ma Q}$
with $p$ the only possible finite prime ramified and such that
$KL_p/K$ is unramified at every finite prime.

Let $p$ be a fixed prime and let $L_p\subseteq \cic p{m_p}$. Lor
any $n\in{\ma N}$, the id\`ele group corresponding to $\cic n{}$ is 
\[
{\mc X}_n=\prod_{i=1}^t U_{p_i}^{(\alpha_i)}\times \prod_{\substack{
q \text{\ prime}\\ q\notin \{p_1,\ldots,p_t\}}} U_q\times {\ma R}^+,
\]
where $n=\prod_{i=1}^t p_i^{\alpha_i}$. As in the case of function
fields, it follows that the id\`ele group corresponding to $L_p$ is
of the form
\[
\Delta_p=H_p\times \prod_{\substack{
q \text{\ prime}\\ q\neq p}} U_q\times {\ma R}^+,
\]
where $U_p^{(m_p)}\subseteq H_p\subseteq U_p$.

We have $e_{K_{\pK_i}|{\ma Q}_p}=e_{K|{\ma Q}}(\pK_i|p)
=e_i$. The extension $L_p/{\ma Q}$ is totally ramified at $p$ and even
we could mix up $L_p$ with the completion of $L$ at $p$.
We have, with both meanings of $L_p$, that
$[L_p:{\ma Q}]=[L_p:{\ma Q}_p]=e_p(L_p|{\ma Q})$.

By Theorem \ref{T2.5} we have that the norm group of
the abelian extension $KL_p/K$ is $\N^{-1}_{K/{\ma Q}}(\Delta_p)$.
Since $L_p$ is maximum, we want $\Delta_p$ to be such that
(see \cite[Corolario 17.6.47]{RzeVil2017})
\begin{gather*}
\N^{-1}_{K/{\ma Q}}(\Delta_p)\subseteq \prod_{\pK\text{\ finite}}
U_{\pK}\times \prod_{\pK\text{\ real}}\*{K_{\pK}}=\prod_{\pK\text{\ finite}}
U_{\pK}\times \prod_{\pK\text{\ real}}\*{\ma R}\subseteq J_K,
\intertext{or}
\Delta_p\subseteq \N_{K/{\ma Q}}\Big(\prod_{\pK\text{\ finite}}
U_{\pK}\times \prod_{\pK\text{\ real}}\*{\ma R}\Big).
\end{gather*}

Let $\vec \alpha\in \prod_{\pK\text{\ finite}} U_{\pK}\times
\prod_{\pK\text{\ real}}\*{\ma R}\subseteq J_K$, $\vec \alpha=
(\alpha_{\pK})_{\pK}$. Then
\[
\N_{K/{\ma Q}}\vec \alpha=\prod_{q\text{\ finite}}\Big(\prod_{\pK|q}
\N_{K_{\pK}/{\ma Q}_p}\alpha_{\pK}\Big)\Big(\prod_{\pK\text{\ real}}
\N_{{\ma R}/{\ma R}}\alpha_{\pK}\Big).
\]

As in the case of function fields we obtain that
\[
H_p=\prod_{\pK|p} \N_{K_{\pK}/{\ma Q}_p}U_{\pK}\quad
\text{and}\quad \Delta_p=\prod_{\pK|p} \N_{K_{\pK}/{\ma Q}_p}U_{\pK}
\times \prod_{\substack{q\text{\ prime}\\q\neq p}}U_q\times
{\ma R}^+.
\]

In other words, let
\begin{gather*}
S_i=\N_{K_{\pK_i}/{\ma Q}_p} U_{\pK_i}
\times \prod_{\substack{q\text{\ prime}\\q\neq p}}U_q\times
{\ma R}^+ \subseteq U_p\times 
\prod_{\substack{q\text{\ prime}\\q\neq p}}U_q\times
{\ma R}^+.
\intertext{We have} 
\Delta_p=\prod_{i=1}^r S_i.
\end{gather*}

Now
$S_i$ corresponds to a field
$R_i\subseteq \cic p{n_p}$ and from 
\cite[Teorema 17.6.49]{RzeVil2017} it follows
that $\prod_{i=1}^r S_i$
corresponds to $\bigcap_{i=1}^r R_i$. Thus $L_p=\bigcap_{
i=1}^r R_i$. Furthermore, since $R_i$
corresponds to $S_i$, we have
\begin{gather*}
[C_{\ma Q}:\*{\ma Q}S_i]=[R_i:{\ma Q}]
\quad\text{and}\quad \Gal(R_i/{\ma Q})\cong C_{\ma Q}/\*{\ma Q}S_i.
\intertext{Since in each field $R_i/{\ma Q}$, $1\leq i\leq r$, the only finite
prime ramified is $p$ and it is totally ramified, the global and the local
degrees are equal so that $[R_i:{\ma Q}]=[(R_i)_{\pK_i}:{\ma Q}_p]$.
On the other hand, since $(R_i)_{\pK_i}/{\ma Q}_p$ is fully ramified
we have}
[(R_i)_{\pK_i}:{\ma Q}_p]= [U_p:\N_{K_{\pK_i}/{\ma Q}_p} 
U_{\pK_i}]
\intertext{(see \cite[Proposici\'on 17.2.15]{RzeVil2017}).
Thus}
[R_i:{\ma Q}]=[C_{\ma Q}:\*{\ma Q}S_i]=[U_p:\N_{K_{\pK_i}/{\ma Q}_p}
U_{\pK_i}],\\
[L_p:{\ma Q}]=\Big[\bigcap_{i=1}^r R_i:{\ma Q}\Big]=
\Big[U_p:\prod_{i=1}^r \N_{K_{\pK_i}/{\ma Q}_p}U_{\pK_i}\Big].
\end{gather*}

When $p\geq 3$, we have that $\cic pm$ is cyclic for every
$m\in{\ma N}$, however, when $p=2$, $\cic 2m$ is not cyclic
for $m\geq 3$. We study the two cases.

Let $G=\langle\sigma\rangle
\cong C_n$ be a finite cyclic group of order $n\in{\ma N}$ and 
let $H_i=\langle \sigma^{j_i}\rangle<G$ where $j_i|n$, $i=1,2$.
Let $H_1\cap H_2=\langle \sigma^t\rangle$ and $H_1H_2=
\langle \sigma^s\rangle$ with $s,t|n$.

We have $\sigma^t\in H_i$, $i=1,2$ so that there exist $a_i\in
{\ma Z}$ such that $\sigma^t=\sigma^{j_ia_i}$, $i=1,2$. Therefore
$t\equiv j_ia_i\bmod n$, $i=1,2$, that is, $t=j_ia_i+l_in$, $i=1,2$.
Hence $j_i|t$, $i=1,2$ so that $\lcm[j_1,j_2]|t$.

Let $u=\lcm[j_1,j_2]$, $j_i|u$. Set $u=j_ib_i$. Then $\sigma^u=
\sigma^{j_ib_i}\in H_i$, $i=1,2$. Thus $\sigma^u\in H_1\cap H_2=
\langle\sigma^t\rangle$ and $\sigma^u=\sigma^{tc}$ for some $c$
and $u=tc+ln$. It follows that $t|u=\lcm[j_1,j_2]$. Therefore $t=u$.

In other words, $H_1\cap H_2=\langle\sigma^{\lcm[j_1,j_2]}\rangle$.

Now, $H_1H_2=\langle\sigma^s\rangle$, $\frac{n}{s}=|H_1H_2|=
\frac{|H_1||H_2|}{|H_1\cap H_2|}=\frac{\frac{n}{j_1}\frac{n}{j_2}}
{\frac{n}{t}}=\frac{nt}{j_1j_2}$. Therefore $st=j_1j_2$ and
$j_1j_2=\gcd(j_1,j_2)\lcm[j_1,j_2]=\gcd(j_1,j_2)t=st$.
Hence $s=\gcd(j_1,j_2)$.

In short, we have

\begin{proposition}\label{P3.10}
Let $G=\langle\sigma\rangle$ be a cyclic group of order $n$ and let
$H_i=\langle\sigma^{j_i}\rangle$ with $j_i|n$, $i=1,2$ be two subgroups
of $G$. Then
$$
H_1\cap H_2=\langle\sigma^{\lcm[j_1,j_2]}\rangle, \quad
H_1H_2 =\langle \sigma^{\gcd(j_1,j_2)}\rangle. \eqno{\fin}
$$
\end{proposition}

\begin{corollary}\label{C3.11} With the conditions of Proposition
{\rm{\ref{P3.10}}}, we have
\begin{gather*}
|H_1\cap H_2|=\frac{|G|}{\lcm[j_1,j_2]},\quad
[G:H_1\cap H_2]=\lcm[j_1,j_2]=\lcm\Big[\frac{|G|}{|H_1|},
\frac{|G|}{|H_2|}\Big],\\
|H_1H_2|=\frac{|G|}{\gcd(j_1,j_2)},\\
 [G:H_1H_2]=\frac{|G|}
{|H_1H_2|}=\gcd(j_1,j_2)=\gcd\big([G:H_1],[G:H_2]\big).
\tag*{$\fin$}
\end{gather*}
\end{corollary}

\begin{corollary}\label{C3.12}
If $p>2$ is a prime number and $H_i<\*{{\ma Z}_p}$, $i=1,2$ are
two subgroups of finite index, then $[\*{{\ma Z}_p}:H_1H_2]=
\gcd\big(\big[\*{{\ma Z}_p}:H_1\big],\big[\*{{\ma Z}_p}:H_2]\big)$.
\end{corollary}

\begin{proof}
We have that $\*{{\ma Z}_p}\cong C_{p-1}\times {\ma Z}_p$ where
$C_{p-1}$ is the cyclic group of order $p-1$. Let $H_i=H_i'\times
p^{n_i}{\ma Z}_p$ where $H_i'$ is the torsion of $H_i$, $i=1,2$.
Then $H_1H_2=H_1'H_2'\times p^{\min\{n_1,n_2\}}{\ma Z}_p$.
Therefore
\begin{align*}
[\*{{\ma Z}_p}:H_1H_2]&=[C_{p-1}:H_1'H_2']p^{\min\{n_1,n_2\}}\\
&=\gcd\big([C_{p-1}:H_1'],[C_{p-1}:H_2']\big)p^{\min\{n_1,n_2\}}\\
&=\gcd\big(\big[\*{{\ma Z}_p}:H_1\big],\big[\*{{\ma Z}_p}:H_2\big]\big).
\end{align*}
\end{proof}

We apply the previous discussion to the case $p>2$.

\begin{proposition}\label{P3.2}
If $p>2$, $\cic p{m_p}/{\ma Q}$ is a cyclic extension and $L_p/
{\ma Q}$ is a cyclic extension. For $F_1, F_2$ contained in
$\cic p{m_p}$, we have
\[
[F_1\cap F_2:{\ma Q}]=\gcd\big([F_1:{\ma Q}],[F_2:{\ma Q}]\big).
\]
\end{proposition}

\begin{proof}
We consider $F_1F_2/{\ma Q}$ which is cyclic since $\cic p{m_p}
/{\ma Q}$ is a cyclic extension. We have
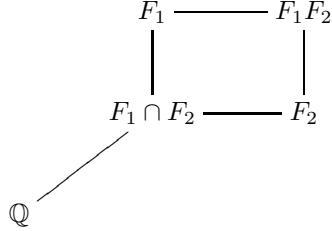
\[
\xymatrix{
&F_1\ar@{-}[d]\ar@{-}[d]\ar@{-}[r]&F_1F_2\ar@{-}[d]\\
&F_1\cap F_2\ar@{-}[r]\ar@{-}[dl]&F_2\\ {\ma Q}
}
\]
Let $a=[F_1\cap F_2:{\ma Q}]$, $b=[F_1:{\ma Q}]$ and
$c=[F_2:{\ma Q}]$. We have that $a|b$ and $a|c$ so that
$a|\gcd(b,c)$. Now, since $\gcd(b,c)|b$ and $\gcd(b,c)|c$,
there exists a unique field $F_0$ satisfying $[F_0:{\ma Q}]=
\gcd(b,c)$, $F_0\subseteq F_1$ and $F_0\subseteq F_2$.
Hence $F_0\subseteq F_1\cap F_2$. This implies $\gcd(b,c)
=[F_0:{\ma Q}]|[F_1\cap F_2:{\ma Q}]=a$. Thus 
$a=\gcd(b,c)$.
\end{proof}

\begin{corollary}\label{C3.3}
With the conditions of Proposition {\rm{\ref{P3.2}}},
if $p>2$ and $F_1,\ldots,F_t\subseteq \cic p{m_p}$, we have
\[
\Big[\bigcap_{i=1}^t F_i:{\ma Q}\Big]=\gcd_{1\leq i\leq t}
\big([F_{i}:{\ma Q}]\big).
\]
\end{corollary}

\begin{proof}
Use induction.
\end{proof}

\begin{remark}\label{R3.4}{\rm{
If $p=2$, Proposition \ref{P3.2} is not longer true. For instance, if
$F_1={\ma Q}(\sqrt{2})$, $F_2={\ma Q}(i)$, then 
$[F_1:{\ma Q}]=[F_2:{\ma Q}]=2$ and since $F_1\cap F_2=
{\ma Q}$, we have $[F_1\cap F_2:{\ma Q}]=1$.
}}
\end{remark}

\begin{remark}\label{R3.5}{\rm{
Since
\begin{gather*}
\Big[\bigcap_{i=1}^r R_{i}:{\ma Q}\Big]=[L_p:{\ma Q}]=
\Big[\*{{\ma Q}_p}:\prod_{i=1}^r\N_{(R_i)_{{\pK_i}}/{\ma Q}_p}
\*{(R_i)_{\pK_i}}\Big]=\Big[U_p:\prod_{i=1}^r
\N_{(R_i)_{{\pK_i}}/{\ma Q}_p} U_{\pK_i}\Big],
\intertext{for $p>2$, we have}
[L_p:{\ma Q}]=\gcd_{1\leq i\leq r}\big([R_i:{\ma Q}]\big)=
\gcd_{1\leq i\leq r}\big([U_{\pK_i}\colon \prod_{i=1}^r
\N_{(R_i)_{{\pK_i}}/{\ma Q}_p} U_{\pK_i}]\big).
\end{gather*}
}}
\end{remark}

The main result on number fields is the following.

\begin{theorem}\label{T3.7}
Let $K/{\ma Q}$ be a finite extension. With the above notations, we have
$\ge K=KL$ where $L=\prod_{p\text{\ finite}}L_p$ and $L_p$ is a
subfield of $\cic p{m_p}$ satisfying
\[
[L_p:{\ma Q}]=\Big[U_p:
\prod_{i=1}^r \N_{K_{\pK_i}/{\ma Q}_p}U_{\pK_i}\Big],
\]
where
$\con_{{\ma Q}/K}p=\pK_1^{e_1}\cdots \pK_r^{e_r}$.

Furthermore, if $p>2$, 
\begin{gather*}
[L_p:{\ma Q}]=
\gcd_{1\leq i\leq r}[U_p:\N_{K_{\pK_i}/{\ma Q}_p}U_{\pK_i}],
\intertext{$L_p$ is determined by its degree $[L_p:{\ma Q}]$ and $L_p$
is the class field of}
\prod_{i=1}^r\N_{K_{\pK_i}/{\ma Q}_p}
U_{\pK_i}\times \prod_{\substack{q\text{\ prime}\\q\neq p}}U_q\times
{\ma R}^+.
\end{gather*}

The tame ramification degree of the extension $[L_p:{\ma Q}]$ is
\[
e^{\rm{tame}}=\gcd(e_1,\ldots,e_r, p-1). 
\]
\end{theorem}

\begin{proof}
It remains to find $e^{\rm{tame}}$. Note that necessarily, $p\geq 3$.
Let $L'_p$ be the subfield of 
$L_p$ of degree $b_p$ where $[L_p:{\ma Q}]=p^{a_p}b_p$ and
$\gcd(p,b_p)=1$.

For any $F\subseteq \cic p{m_p}$ with $[F:{\ma Q}]=d$ and 
$\gcd (p,d)=1$, $F/{\ma Q}$ is tamely ramified. Assume that
$KF/K$ is unramified at $p$.
\[
\xymatrix{
K\ar@{-}[r]\ar@{-}[d]&KF\ar@{-}[d]\\ {\ma Q}\ar@{-}[r]&F}
\]

By Abhyankar Lemma, if $\pL$ is a prime in $KF$ with
$\pL\cap {\ma Q}=(p)$, $\pL\cap K=\pK_i$, ${\mc Q}=\pL
\cap F$, then
\[
e(\pL|p)=\lcm[e(\pK_i|p),e({\mc Q}|p)]=\lcm[e_i,d].
\]

Therefore $e(\pL|\pK_i)=\frac{e(\pL|p)}{e(\pK_i|p)}=
\frac{e(\pL|p)}{e_i}$, that is, $\pL$ is unramified 
in $KF/K$ if and only if
$e(\pL|\pK_i)=1$ if and only if $e(\pL|p)=e_i$ if and only if
$d|e_i$. Hence, $KF/K$ is unramifed at every finite prime,
if and only if $d|e_i$ for $1\leq i\leq r$ if and only if $d|
\gcd(e_1,\ldots,e_r)$. Since $d|p-1$, this is equivalent to
$d|\gcd(e_1,\ldots,e_r,p-1)$. It follows that $b_p=
\gcd(e_1,\ldots,e_r,p-1)$.
\end{proof}

\begin{remark}\label{R5.5N}{\rm{
Theorem \ref{T3.7} was proved by M. Bhaskaran in \cite{Bh79}
and by X. Zhang in \cite{Xi86}.
}}
\end{remark}

\subsection{Remarks on $L_2$}\label{S4.1}

For any finite extension $K/{\ma Q}$, we have that if $\ge K=
KL$ with $L/{\ma Q}$ the maximum abelian extension contained
in $K_{H^+}$, we have proved that if $L=\prod_{q\text{\ prime}}
L_q$, then for $p\geq 3$, $L_p$ is completely determined by
\[
[L_p:{\ma Q}]=\Big[ U_p:\prod_{\pK|p}\N_{K_{\pK}/{\ma Q}_p}
U_{\pK}\Big]=\gcd_{\pK|p}[U_p:\N_{K_{\pK}/{\ma Q}_p}U_{\pK}].
\]
This is not so for $p=2$. We want to study $L_2$.

Let $[L_2:{\ma Q}]=2^a$, $a\geq 1$. For $a\geq 2$, there are
three possible $L_2$, namely
\[
\cic 2{a+1},\quad \cic 2{a+2}^+={\ma Q}(\zeta_{2^{a+2}}
+\zeta_{2^{a+2}}^{-1})\quad\text{and}\quad \cic 2{a+2}^-:=
{\ma Q}(\zeta_{2^{a+2}}-\zeta_{2^{a+2}}^{-1}),
\]
see \cite[\S 5.3.1]{RzeVil2017}.

If $L_2$ is real, then $L_2=\cic 2{a+2}^+$. If $L_2$ has conductor
$2^{a+1}$ then $L_2=\cic 2{a+1}$. In other words, $L_2$ can be
determined by means of its conductor and whether it is real or not.

If $K(\zeta_{2^{a+1}})/K$ is unramified, we have $L_2=
\cic 2{a+1}$. In any case $\cic 2{a+1}^+\subseteq L_2$, and
therefore $K\cic 2{a+1}^+/K$ is unramified.
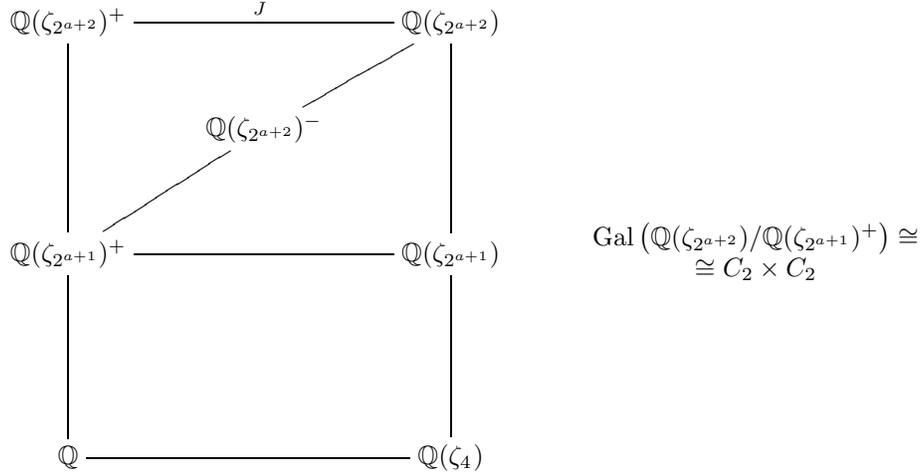
\[
\xymatrix{
\cic 2{a+2}^+\ar@{-}[rr]^J\ar@{-}[dd]&&\cic 2{a+2}\ar@{-}[dd]\\
&\cic 2{a+2}^-\ar@{-}[ur]\ar@{-}[dl]\\
\cic 2{a+1}^+\ar@{-}[rr]\ar@{-}[dd]&&\cic 2{a+1}\ar@{-}[dd]&
{\begin{array}{c}
\Gal\big(\cic 2{a+2}/\cic 2{a+1}^+\big)\cong\\ \cong C_2\times C_2
\end{array}}\\
\\
{\ma Q}\ar@{-}[rr]&&\cic 4{}}
\]

We need to determine the group of id\`eles corresponding to
each extension: $L_2\in\big\{\cic 2{a+1},\cic 2{a+2}^+,\cic 2{a+2}^-\big\}$.

Recall that for a local field $K$ we have $\*K\cong \*\F\times U_{\pK}^{(1)}\times
(\pi)$, where $\pi$ is a uniformizer element, $v_{\pK}(\pi)=1$, $U_{\pK}$
are the units of $\*K$, $U_{\pK}^{(1)}$ are the units modulo $1$,
that is, $U_{\pK}^{(1)}=\{\xi\in U_{\pK}\mid \xi-1\in (\pi)\}=1+\pi {\mc O}_K
=1+\pK$, and $U_{\pK}^{(1)}\times \*\F=U_{\pK}$ where $\F$ is the
residue field.

In the particular case of $K=\*{{\ma Q}_p}$, $q=p$, $\*{{\ma F}_p}\cong
C_{p-1}={\ma Z}/(p-1){\ma Z}$ and $U_p=\Big\{\sum_{i=0}^{\infty}
a_ip^i\mid a_0\neq 0, a_i\in\{0,1,\ldots,p-1\}\text{\ for all\ }i\Big\}
\cong \*{{\ma Z}_p}$, where ${\ma Z}_p$ denotes the ring of $p$--adic
integers and $\*{{\ma Z}_p}$ is the multiplicative group of ${\ma Z}_p$.

We have

\begin{proposition}\label{P3.7}
\las
\item If $p>2$, $\*{{\ma Z}_p}\cong C_{p-1}\times {\ma Z}_p$ as groups.

\item If $p=2$, $1+2{\ma Z}_2\cong \{\pm 1\}\times (1+4{\ma Z}_2)$ and
$1+4{\ma Z}_2\cong {\ma Z}_2$. In particular,
$$
U_2=U_2^{(1)}=\*{{\ma Z}_2}\cong 1+2{\ma Z}\cong\{\pm 1\}\times
(1+4{\ma Z}_2)\cong \{\pm 1\}\times {\ma Z}_2. \eqno{\fin}
$$
\end{list}
\end{proposition}

We are going to identify complex conjugation $J$ with $-1$ since
$J(\zeta_{2^n})=\zeta_{2^n}^{-1}$ for all $n$.

The non--zero closed subgroups of $U_2=\*{{\ma Z}_2}
\cong\{\pm 1\}\times
{\ma Z}_2$ are: $\{\pm 1\}\times 2^n{\ma Z}_2$, $2^n{\ma Z}_2$
and $\{\pm 1\}\cdot 2^n{\ma Z}_2$ with $n\in{\ma N}\cup \{0\}$.

The quotient groups are respectively
\las
\item[$\bullet$] $\frac{\{\pm 1\}\times {\ma Z}_2}{\{\pm 1\}\times 2^n{\ma Z}_2}
\cong \frac{{\ma Z}_2}{2^n {\ma Z}_2}\cong C_{2^n}$,

\item[$\bullet$] $\frac{\{\pm 1\}\times {\ma Z}_2}{2^n{\ma Z}_2}\cong
\{\pm 1\}\times \frac{{\ma Z}_2}{2^n{\ma Z}_2}\cong \{\pm 1\}
\times C_{2^n}$,

\item[$\bullet$] $\frac{\{\pm 1\}\times {\ma Z}_2}{\{\pm 1\}\cdot 2^n{\ma Z}_2}
\cong {\mc H}$.
\end{list}

Let us study ${\mc H}$. Consider $b:=1\in{\ma Z}_2$.
Then $b$ is a topological
generator of ${\ma Z}_2$. Let $a:=-1$ be the unique torsion element
of $\*{{\ma Z}_2}$ of order $2$. Let ${\mc H}$ be the procyclic group
with topological generator $ab^{2^n}$: ${\mc H}=
\overline{\langle ab^{2^n}\rangle}$ (topological closure). Denote by
$\tilde{a}$ and $\tilde{b}$ the classes of $a$ and $b$ modulo 
${\mc H}$ respectively: $\tilde{a}=a\bmod {\mc H}$; $\tilde{b}=
b\bmod {\mc H}$.

We have ${\mc G}/{\mc H}=\langle\tilde{a},\tilde{b}\rangle$ where
${\mc G}=\{\pm 1\}\times {\ma Z}_2\cong\*{{\ma Z}_2}$. Since $ab^{2^n}\in
{\mc H}$, $\tilde{b}^{2^n}=\tilde{a}^{-1}\bmod{\mc H}$ and $\tilde{a}^{-1}
=\tilde{a}\bmod {\mc H}$ (indeed, $a^{-1}=a=-1$). Therefore 
${\mc G}/{\mc H}=\langle\tilde{b}\rangle$ since $\tilde{a}=\tilde{b}^{2^n}\in
\langle\tilde{b}\rangle$ so that ${\mc G}/{\mc H}$ is a cyclic group.

Note that $b^{2^n}\notin {\mc H}$ since otherwise $a\in {\mc H}$ but
$a$ is a torsion element and ${\mc H}$ is torsion free. Therefore
$b^{2^n}\notin{\mc H}$. On the other hand $b^{2^{n+1}}=b^{2^n}
b^{2^n}\equiv ab^{2^n}\bmod {\mc H}$ so that $b^{2^{n+1}}\in
{\mc H}$. It follows that $o(\tilde{b})=2^{n+1}$. Thus ${\mc G}/{\mc H}$
is a cyclic group of order $2^{n+1}$.

Uniformizing the indexes, we have
\las
\item[$\bullet$] $\frac{\{\pm 1\}\times {\ma Z}_2}{\{\pm 1\}\times 2^m{\ma Z}_2}
= \frac{\overline{\langle a, b\rangle}}{\overline{\langle a,b^{2^m}\rangle}}
=\langle b\bmod b^{2^m}\rangle \cong C_{2^m}$,

\item[$\bullet$] $\frac{\{\pm 1\}\times {\ma Z}_2}{2^{m-1}{\ma Z}_2}=
\frac{\overline{\langle a, b\rangle}}{\overline{\langle b^{2^{m-1}}\rangle}}
=\langle\tilde{a},\tilde{b}\rangle \bmod b^{2^{m-1}}\cong C_2\times
C_{2^{m-1}}$,

\item[$\bullet$] $\frac{\{\pm 1\}\times {\ma Z}_2}{\{\pm 1\}\cdot 2^m{\ma Z}_2}
=\frac{\overline{\langle a, b\rangle}}{\overline{\langle ab^{2^{m}}\rangle}}=
\langle\tilde{b}\rangle\bmod {\mc H}=C_{2^n}$.
\end{list}

Define ${\mc A}_m:=\{\pm 1\}\times 2^m{\ma Z}_2$; 
${\mc B}_m:=2^{m-1}{\ma Z}_2$; ${\mc C}_m:=\{\pm 1\} 
2^{m-1}{\ma Z}_2$.

We have 
\las
\item[$\bullet$] ${\mc R}_m:={\mc G}/{\mc A}_m\cong \Gal(\cic 2{m+2}^+/
{\ma Q})\cong C_{2^m}$ since $-1\in{\mc A}_n$,

\item[$\bullet$] ${\mc S}_m:={\mc G}/{\mc B}_m\cong \Gal(\cic 2{m+1}/
{\ma Q})\cong C_2\times C_{2^{m-1}}$ since ${\mc G}/{\mc B}_n$ is
noncyclic,

\item[$\bullet$] ${\mc T}_m:={\mc G}/{\mc C}_m\cong \Gal(\cic 2{m+2}^-/
{\ma Q})\cong C_{2^m}$ since it is cyclic and $-1\notin{\mc C}_n$.

\end{list}

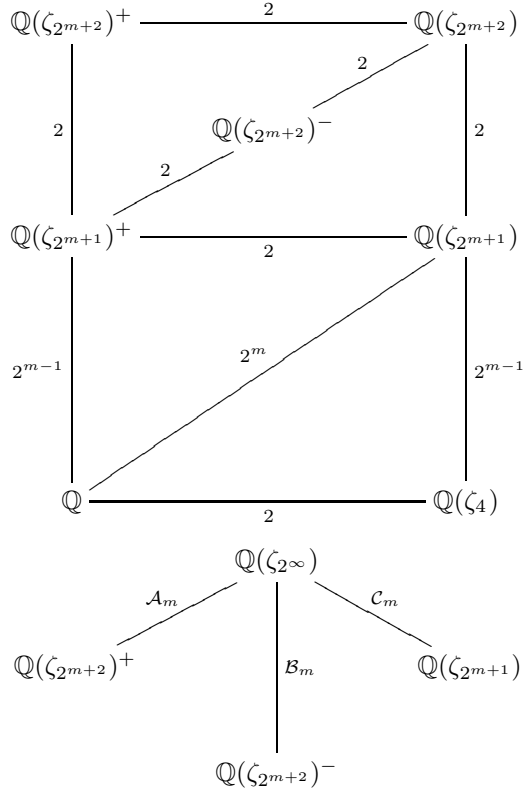
\begin{gather*}
\xymatrix{
\cic 2{m+2}^+\ar@{-}[rr]^2\ar@{-}[dd]_2&&\cic 2{m+2}\ar@{-}[dd]^2\\
&\cic 2{m+2}^-\ar@{-}[ur]^2\ar@{-}[dl]_2\\
\cic 2{m+1}^+\ar@{-}[rr]_2\ar@{-}[ddd]_{2^{m-1}}&&\cic 2{m+1}
\ar@{-}[ddd]^{2^{m-1}}\\
\\
\\
{\ma Q}\ar@{-}[rr]_2\ar@{-}[uuurr]^{2^m}&&\cic 4{}}\\
\xymatrix{
&\cic 2{\infty}\ar@{-}[dl]_{{\mc A}_m}\ar@{-}[dd]^{{\mc B}_m}
\ar@{-}[dr]^{{\mc C}_m}\\
\cic 2{m+2}^+&&\cic 2{m+1}\\ &\cic 2{m+2}^-}
\end{gather*}

Since $[L_2:{\ma Q}]=2^m$ and $\big[U_2:\prod_{\pK|2}
U_{\pK}\big]=2^m$, it follows the following theorem.

\begin{theorem}\label{T3.8} If $[L_2:{\ma Q}]=2^m$, then
\las
\item $L_2=\cic 2{m+2}^+\iff$ for every place $\pK$ of $K$ with
$\pK|2$ we have $-1\in\N_{K_{\pK}/{\ma Q}_2} U_{\pK}$, that is,
$-1\in \bigcap_{\pK|2}\N_{K_{\pK}/{\ma Q}_2} U_{\pK}$.

\item $L_2=\cic 2{m+1}\iff \bigcap_{\pK|2}\N_{K_{\pK}/{\ma Q}_2}
U_{\pK}$ is not cyclic (automatically we have that $-1\notin
\bigcap_{\pK|2}\N_{K_{\pK}/{\ma Q}_2}$).

\item $L_2=\cic 2{m+2}^-\iff \bigcap_{\pK|2}\N_{K_{\pK}/{\ma Q}_2}$
is cyclic and $-1\notin \bigcap_{\pK|2}\N_{K_{\pK}/{\ma Q}_2}$. $\fin$
\end{list}
\end{theorem}

\section{Some remarks on genus fields of number
fields}\label{S5}

Let $L/{\ma Q}$ be a finite Galois extension. Since $L/{\ma Q}$
is normal, $L$ is either totally real or totally imaginary. Let $J\colon
{\ma C}\lra{\ma C}$ be the complex conjugation. Since $J|_{\ma Q}=
\Id_{\ma Q}$ and $L/{\ma Q}$ is normal, we have $J(L)=L=\bar{L}$.
Hence $J|_L\in G:=\Gal(L/{\ma Q})$. Furthermore $J|_L$ has order
$o(J|_L)=1$ or $2$. Let $L^J$ be the fixed field of $L$ under the
action of $J$. We have $\Gal(L|L^J)=\langle J|_L\rangle \cong
\{1\}$ or $C_2$, the cyclic group of order $2$ and $[L:L^J]|2$.
Furthermore, $L^J\subseteq {\ma R}$.

Note that $L^J$ is neither necessarily normal over ${\ma Q}$ nor
totally real. For instance, if $L={\ma Q}(\zeta_3,\sqrt[3]{2})$.
\[
\xymatrix{
{\ma Q}(\sqrt[3]{2})\ar@{-}[rrr]^{C_2}_{\alpha}\ar@{-}[d]&&&
L={\ma Q}(\zeta_3,\sqrt[3]{2})\ar@{-}[d]_{\beta}^{C_3}\ar@{-}[dlll]\\
{\ma Q}\ar@{-}[rrr]&&&{\ma Q}(\zeta_3)}
\]
Then $\Gal(L/{\ma Q})=\langle\alpha,\beta\rangle=C_2\ltimes C_3\cong
S_3$, the symmetric group in $3$ elements. $L$ is totally imaginary
and $L^J={\ma Q}(\sqrt[3]{2})$, the extension ${\ma Q}(\sqrt[3]{2})/{\ma Q}$
is not normal and the $3$ embeddings are $\sqrt[3]{2}\lra
\begin{cases}\sqrt[3]{2}\\ \zeta_3\sqrt[3]{2}\\ \zeta_3^2\sqrt[3]{2}
\end{cases}$. In other words, with the usual meaning, $r_1=1$
and $r_2=1$.

When $L/{\ma Q}$ is abelian, then $\langle J|_L\rangle\lhd G$, $L^J/
{\ma Q}$ is a Galois extension and $L^J$ is totally real.

In the case of genus fields, we consider $K/{\ma Q}$ a finite extension
and let $K_H$ and $K_{H^+}$ be the Hilbert class field and the Hilbert
extended (narrow) class field of $K$ respectively. Then the genus field
$\g K$ is the maximum extension such that $K\subseteq \g K\subseteq
K_H$ with $\g K=KF$, $F/{\ma Q}$ abelian. In particular, $F=\g F$. The 
extended or narrow genus field $\ge K$ of $K$ is
the maximum extension such that $K\subseteq \ge K
\subseteq K_{H^+}$ with $\ge K=KL$ and $L/{\ma Q}$ is abelian.
In particular, $\ge {L}=L$. Recall that $\ge {L}$ is the maximum abelian
extension of ${\ma Q}$ such that $\ge {L}/L$ is unramfied at every finite prime
and $\g F$ is the maximum abelian extension of ${\ma Q}$ with $\g F/K$ 
unramified at every prime.

From the remarks above, it follows that $[\ge K:\g K]=1$ or $2$ for
every finite abelian extension $K/{\ma Q}$. Now,
we have $K_H\subseteq K_{H^+}$
and in fact $\Gal(K_{H^+}/K_H)\cong C_2^r$ for some $r\in{\ma N}
\cup \{0\}$. In our notation, we have that $F\subseteq L$ since
$KF/K$ is unramified and $F/{\ma Q}$ is abelian. On the other hand,
$L^J$ is totally real, $L^J/{\ma Q}$ is abelian and $KL^J/K$
is unramified at every prime. It follows that
\begin{gather*}
L^J\subseteq F\subseteq L.
\intertext{Since $F=\g F$ it follows that $[L:F]|2$ and therefore}
[\ge K:\g K]=[KL:KF]|[L:F]=1\text{\ or\ }2.
\end{gather*}

In short, we have

\begin{proposition}\label{P4.1}
For a finite extension $K/{\ma Q}$, we have $[\ge K:\g K]|2$. $\fin$
\end{proposition}

Now consider $K_i/{\ma Q}$, $i=1,2$, two finite extensions and let $K=K_1K_2$.
We have $K_i\subseteq K$ for $i=1,2$. On the other hand $\g{(K_1)}/K_1$ is 
unramified and abelian, it follows that $K\g{(K_1)}/KK_1=K$ is unramified
and abelian. Hence $K\g{(K_1)}\subseteq \g K$. It follows that $\g{(K_1)}
\subseteq \g K$. Similarly $\g{(K_2)}\subseteq \g K$. Therefore
$\g{(K_1)}\g{(K_2)}\subseteq \g K$.

\begin{remark}\label{R4.2}{\rm{
Not necessarily $\g{(K_1)}\g{(K_2)}=\g K$.
}}
\end{remark}

\begin{example}\label{E4.3}{\rm{
Let $p,q,p_1,q_1$ be four odd distinct primes. Let $K_1 =
\cic {pq}{}^+$, $K_2=\cic {p_1q_1}{}^+$. Then, using Dirichlet
characters, we have that $\g{(K_1)}\subseteq \cic {pq}{}$ and
$\cic {pq}{}/\cic {pq}{}^+$ is ramified at $\infty$, it follows
that $\g{(K_1)}=K_1$. Similarly $\g{(K_2)}=K_2$.

Furthermore, since $p\neq q$ (respectively $p_1\neq q_1$),
$\cic {pq}{}/\cic {pq}{}^+$ is ramified only at $\infty$, that is,
$\cic {pq}{}/\cic {pq}{}^+$ is unramified at every finite prime
(\cite[Teorema 5.3.2]{RzeVil2017}).

Now $K_1K_2=K=\cic {pq}{}^+\cic {p_1q_1}{}^+\subseteq
\cic {pqp_1q_1}{}$.
\begin{gather*}
\xymatrix{
\cic {pq}{}\ar@{-}[rrr]\ar@{-}[dd]_2&&& \cic {pqp_1q_1}{}\ar@{-}[dddd]\\
&&\cic {pqp_1q_1}{}^+\ar@{-}[ru]^2\ar@{-}[dl]_2\\
K_1=\cic {pq}{}^+\ar@{-}[r]\ar@{-}[dd]&K_1K_2=K\ar@{-}[dd]\\
\\
{\ma Q}\ar@{-}[r]&K_2=\cic {p_1q_1}{}^+\ar@{-}[rr]_2&&
\cic {p_1q_1}{}}
\intertext{We have that $\cic {pqp_1q_1}{}^+/K$ is unramified since $p$
is unramified in $\cic {pq}{}/\cic{pq}{}^+$ and thus}
e_p(\cic{pqp_1q_1}{}|{\ma Q})=p-1=e_p(\cic{pq}{}^+|{\ma Q})=e_p(K|{\ma Q}).
\end{gather*}

The same holds for $q,p_1$ and $q_1$. Now, $\infty$ is ramified
in $\cic {pqp_1q_1}{}/\cic{pqp_1q_1}{}^+$.
It follows that $\g K=\cic {pqp_1q_1}{}^+$ and that $[\g K:\g{(K_1)}
\g{(K_2)}]=2>1$.
}}
\end{example}

\begin{remark}\label{R4.4-1}{\rm{
For any two finite abelian extensions $K_i/{\ma Q}$, $i=1,2$ we
have $\ge K=\ge{(K_1)}\ge{(K_2)}$ where $K=K_1K_2$ (see
\cite{BaMoReRzVi2018}).
}}
\end{remark}

\begin{theorem}\label{T4.4} Let $K_i/{\ma Q}$, $i=1,2$ be two
finite abelian extensions and let $K=K_1K_2$. Then
\[
[\g K:\g{(K_1)}\g{(K_2)}]|2.
\]
\end{theorem}

\begin{proof}
In general we consider a finite abelian extension $K/{\ma Q}$.
Let $L=\ge K$. We have $\g K=L^+ K$ (see \cite{BaMoReRzVi2018}).
Let $K=K_1K_2$. Then $\ge K=\ge{(K_1)}\ge{(K_2)}$. 
Therefore $L=L_1L_2$ and $\g K=L^+K$,
$\g {(K_1)}\g{(K_2)}=L_1^+K_1L_2^+K_2=L_1^+L_2^+K$. Hence
\[
[\g K:\g{(K_1)}\g{(K_2)}]=[L^+K:L_1^+L_2^+K]|[L^+:L_1^+L_2^+].
\]

To prove the result, it suffices to show that for two finite abelian extensions
$L_i/{\ma Q}$, $i=1,2$, and for $L=L_1L_2$, we have
$[L^+:L_1^+L_2^+]|2$.

In general, we have $L^+=L\cap \cic n{}^+=L\cap \cic n{}^J$ for
$L\subseteq \cic n{}$. In particular, if $S:=\Gal(\cic n{}/L)$, $L^+=L\cap
\cic n{}^+=\cic n{}^S\cap \cic n{}^I=\cic n{}^{SI}$ where $I=\langle
J\rangle$ and thus $\Gal(\cic n{}/L^+)=SI$.

Let $S_i:=\Gal(\cic n{}/L_i)$, $i=1,2$. Since $L=L_1L_2$, we have
$S=S_1\cap S_2$. We also have
\begin{gather}
L_1^+L_2^+=\cic n{}^{S_1I}\cic n{}^{S_2I}=\cic n{}^{S_1I\cap S_2I}
\subseteq L^+=\cic n{}^{SI}.\nonumber
\intertext{Therefore}
\Gal(L^+/L_1^+L_2^+)\cong\frac{\Gal(\cic n{}/L_1^+L_2^+)}{\Gal(\cic n{}/L^+)}
\cong \frac{S_1I\cap S_2I}{SI}=\frac{S_1I\cap S_2I}{(S_1\cap S_2)I}.\label{Eq4.1}
\end{gather}

Now
\begin{align*}
|S_1I\cap S_2I|&=\frac{|S_1I||S_2I|}{|S_1S_2I|}=\frac{\frac{|S_1||I|}{
|S_1\cap I|}\frac{|S_2||I|}{|S_2\cap I|}}{\frac{|S_1S_2||I|}{|S_1S_2\cap I|}}
=\frac{\frac{|S_1||S_2||I|^2}{|S_i\cap I||S_2\cap I|}}{\frac{|S_1||S_2|}{
|S_1\cap S_2|}\frac{|I|}{|S_1S_2\cap I|}}\\
&=\frac{|S_1\cap S_2||S_1S_2\cap I|}{|S_1\cap I||S_2\cap I|}|I|.
\end{align*}
On the other hand $|(S_1\cap S_2)I|=|SI|=\frac{|S||I|}{|S\cap I|}$.
It follows that
\[
[S_1I\cap S_2I:(S_1\cap S_2)I]=\frac{|S_1\cap S_2||S_1S_2\cap I|}{
|S_1\cap I||S_2\cap I|}\frac{|S\cap I|}{|S||I|}|I|=
\frac{|S_1S_2\cap I||S\cap I|}{|S_1\cap I||S_2\cap I|}.
\]

Now $S\cap I\subseteq S_2\cap I$. Let $\alpha=[S_2\cap I:S\cap I]
\in{\ma N}$. Then
\begin{gather*}
[S_1I\cap S_2I:(S_1\cap S_2)I]=\frac{1}{\alpha}\frac{|S_1S_2\cap I|}
{|S_1\cap I|}=\frac{1}{\alpha}\frac{\frac{|S_1S_2||I|}{|S_1S_2I|}}
{\frac{|S_1||I|}{|S_1I|}}=\frac{1}{\alpha}\frac{|S_1S_2||S_1I|}{
|S_1S_2I||S_1|}.
\end{gather*}

We have $S_1S_2\subseteq S_1S_2I$. Let $\beta=[S_1S_2I:S_1S_2]
\in{\ma N}$. It follows that
\begin{gather*}
[S_1I\cap S_2I:(S_1\cap S_2)I]=\frac{1}{\alpha\beta}\frac{|S_1I|}{|S_1|}=
\frac{1}{\alpha\beta}\frac{|S_1||I|}{|S_1||S_1\cap I|}=\frac{1}{\alpha\beta}
\frac{|I|}{|S_1\cap I|}=\frac{\gamma}{\alpha\beta},
\end{gather*}
with $\gamma=[I:S_1\cap I]||I|$. Therefore
$[S_1I\cap S_2I:(S_1\cap S_2)I]=\frac{\gamma}{\alpha\beta}\in{\ma N}$
and $[S_1I\cap S_2I:(S_1\cap S_2)I]||I|=1$ or $2$. It follows
that
\[
[L^+:L_1^+L_2^+]|2\quad\text{and}\quad [\g K:\g{(K_1)}\g{(K_2)}]|2.
\]
\end{proof}

\end{document}